\documentclass[12pt,reqno,dvipsnames]{amsart}

\usepackage[english]{babel}
\usepackage{microtype}
\usepackage{mathtools,thmtools}
\usepackage{color, graphicx}
\usepackage[all]{xy}
\usepackage{etoolbox}

  \usepackage[utf8]{inputenc}
  \usepackage[left=3cm, right=2.5cm, top=3cm, bottom=3cm]{geometry}
  \usepackage[normalem]{ulem}
  \usepackage[inline]{enumitem}
  \usepackage{amsmath, amsthm, amsfonts, amssymb}
  \usepackage{mathtools}
  \usepackage{thmtools}
  \usepackage[mathscr]{eucal}
  \usepackage[colorlinks=true,citecolor=blue]{hyperref} 
  \usepackage{lmodern}
  \usepackage[capitalise]{cleveref}
  \usepackage{tikz}
  \usepackage{tikz-cd}
  \usetikzlibrary{%
    arrows,
    arrows.meta,
    calc,
    graphs,
    decorations.pathmorphing,
    decorations.markings,
    external
  }

  \theoremstyle{plain}
  
  \newtheorem{proposition}{Proposition}[section]
  \newtheorem{theorem}[proposition]{Theorem}
  \newtheorem{lemma}[proposition]{Lemma}
  \newtheorem{corollary}[proposition]{Corollary}
	\declaretheorem[name=Theorem,qed={\tiny$\blacksquare$},numbered=no]{theorem*}

  \theoremstyle{remark}

  \theoremstyle{definition}
  \newtheorem{definition}[proposition]{Definition}

  \declaretheorem[ name=Example, style=definition, qed=$\triangle$, sharenumber=proposition ]{example}
  \declaretheorem[ name=Remark, style=definition, qed=$\triangledown$, sharenumber=proposition ]{remark}

  \numberwithin{equation}{section}
  \crefname{example}{Example}{Example}
  \Crefname{example}{Example}{Example}
  \crefname{equation}{}{}
  \Crefname{equation}{Equation}{Equations}
  \newcommand{\oversetEq}[1][ ]{
      \stackrel{\mathmakebox[\widthof{=}]{#1}}{=}
  }
  \makeindex

  \newcommand{\inv}{^{-1}}
  \DeclareMathOperator{\id}{id}
  \newcommand{\qtextq}[1]{\quad\text{#1}\quad}
  \newcommand{\qandq}{\qtextq{and}}
  \newcommand{\To}{\Rightarrow}

  \newcommand{\Set}{\mathsf{Set}}
  \newcommand{\cat}[1][C]{{\mathcal{#1}}}
  \newcommand{\tensor}{\otimes}
  \newcommand{\tensUnit}{\mathbf{1}}
  \DeclareMathOperator{\End}{End}
  \DeclareMathOperator{\ev}{ev}
  \DeclareMathOperator{\coev}{coev}

  \newcommand{\sweedler}[1]{_{(#1)}}
  \newcommand{\op}{^{\textup{op}}}

  \newcommand{\counit}{\varepsilon}
  \NewDocumentCommand{\hmodM}{O{H}O{}O{}O{}}{\prescript{#4}{#1}{\mathcal{M}}_{#2}^{#3}}
	
\makeatletter
\patchcmd{\@setaddresses}{\indent}{\noindent}{}{}
\patchcmd{\@setaddresses}{\indent}{\noindent}{}{}
\patchcmd{\@setaddresses}{\indent}{\noindent}{}{}
\patchcmd{\@setaddresses}{\indent}{\noindent}{}{}
\makeatother

\def\ot{\otimes}

\DeclareRobustCommand{\SkipTocEntry}[5]{}
\crefname{part}{\S}{\S\S}
\crefname{chapter}{\S}{\S\S}
\crefname{section}{\S}{\S\S}
\crefname{subsection}{\S}{\S\S}

\title{Everybody knows what a normal gabi-algebra is}

\author{Johannes Berger}
\author{Paolo Saracco}
\address{D\'epartement de Math\'ematique, Universit\'e Libre de Bruxelles, Boulevard du Triomphe, B-1050 Brussels, Belgium.}
\email{paolo.saracco@ulb.be}
\urladdr{\url{https://sites.google.com/view/paolo-saracco}}
\urladdr{\url{https://paolo.saracco.web.ulb.be}}

\author{Joost Vercruysse}
\address{D\'epartement de Math\'ematique, Universit\'e Libre de Bruxelles, Boulevard du Triomphe, B-1050 Brussels, Belgium.}
\email{joost.vercruysse@ulb.be}
\urladdr{\url{http://joost.vercruysse.web.ulb.be}}

\date{\today}
\begin{document}

\allowdisplaybreaks

\begin{abstract}
Let $A$ be a $k$-algebra over a commutative ring $k$. By the renowned Tannaka-Kre{\u\i}n reconstruction, liftings of the monoidal structure from $\hmodM[k]$ to $\hmodM[A]$ correspond to bialgebra structures on $A$ and liftings of the closed monoidal structure correspond to Hopf algebra structures on $A$. In this paper, we determine conditions on $A$ that correspond to liftings of the closed structure alone, {i.e.} without considering the monoidal one, which lead to the notion of what we call a \emph{gabi-algebra}. 
First, we tackle the question from the general perspective of monads, then we focus on the set-theoretic and the linear setting.
Our main and most surprising result is that a \emph{normal gabi-algebra}, that is an algebra $A$ whose category of modules is (associative and unital normal) closed with closed forgetful functor to $\hmodM[k]$, is automatically a Hopf algebra (thus justifying our title).
\end{abstract}

\maketitle

\tableofcontents

\section{Introduction}

A well-known result in the theory of Hopf algebras states that there is a bijective correspondence between bialgebra structures on an algebra $A$ over a commutative ring $k$ and monoidal structures on its category of left modules $\hmodM[A]$ which are lifted along the forgetful functor $\omega \colon \hmodM[A] \to \hmodM[k]$, that is to say, monoidal structures on $\hmodM[A]$ for which $\omega$ is a strict monoidal functor (see \cite{Pareigis}). 
Moreover, the category of modules over a bialgebra $A$ is always a \emph{right closed monoidal} category, in the sense that for every left $A$-module $N$, the endofunctor $- \otimes N \colon \hmodM[A] \to \hmodM[A]$ has a right adjoint given by $\hmodM[A](.A \otimes .N,-)$, where $.A \otimes .N$ is an $A$-bimodule with respect to $x\cdot (y \otimes n) \cdot z = x_1yz \otimes x_2\cdot n$ for all $x,y,z\in A$ and $n\in N$ (see e.g.\ \cite{Schauenburg-Hopf}). In fact, it is also {\em left closed monoidal}, meaning that the endofunctor $N\ot-$ has a right adjoint given by $\hmodM[A](.N \otimes .A,-)$. Hence, $\hmodM[A]$ is always a \emph{closed monoidal} category. Remark, however, that the right closed structure (i.e. the ``right internal homs'') in $\hmodM[A]$ differs in general from that in $\hmodM[k]$. In fact, the forgetful functor $\omega$ lifts not only the monoidal, but the full right closed monoidal structure if and only if also $\hmodM[k](N,-)$ (endowed with some $A$-module structure) provides a right adjoint to $- \otimes N \colon \hmodM[A] \to \hmodM[A]$. If this is the case, then the two right adjoints are related via the canonical isomorphism induced by $.A \otimes N \to .A \otimes .N,\ x \otimes n \mapsto x_1 \otimes x_2\cdot n$. For $N = A$, this map is exactly the so-called {\em Galois} or {\em fusion} map of the bialgebra $A$, whose bijectivity corresponds to the existence of an antipode for $A$. This makes it clear why lifting the right closed monoidal structure is equivalent to having a Hopf algebra structure on $A$. Furthermore, a similar argument tells that the left closed structure is lifted along $\omega$ if and only if the bialgebra has an {\em opantipode}. Thus, $A$ is a Hopf algebra with bijective antipode if and only if both left and right closure lift along $\omega$. 

The above observations form the starting point of the Tannaka-Kre{\u\i}n reconstruction theory for Hopf algebras (or quantum groups), where one aims at reconstructing algebraic structures and properties of (co)algebras from categorical structures and properties of their categories of (co)modules. Such reconstruction theorems have been widely studied at different levels of generality: see e.g.\ \cite{Schauenburg-Tannaka,Ulbrich} for the case of Hopf algebras, \cite{Majid,Saracco-Tannaka} for the case of coquasi-bialgebras (with preantipode), \cite{Schauenburg-bialgebroids,Schauenburg-Hopf} for bialgebroids and Hopf algebroids, \cite{BFVV} for oplax Hopf algebras, \cite{BLV-hopf_monads} for Hopf monads and \cite{Ver} for an overview. 
However, any one of the aforementioned cases follows the foregoing structure, where first a bialgebra-type structure is obtained from lifting the monoidal structure and then additional Hopf-like structures arise from closeness or rigidity of this monoidal structure.

In this paper, we change perspective and we focus on closed structures as starting point. Indeed, closed structures on a category can be defined independently of monoidal ones: see for instance
\cite{Eilenberg-Kelly,Street-skew_closed,UVZ_Eilenberg-Kelly_reloaded}. In a nutshell, a
closed category is a category $\cat$ together with a distinguished object $\tensUnit$ and
a bifunctor $[-,-]\colon \cat\op \times \cat \to \cat$, which can be understood as the
prototype of internal homs, satisfying suitable compatibilities. From this perspective, we may equally look at $\hmodM[k]$ as a closed
category for which the endofunctor $\hmodM[k](V,-)\colon \hmodM[k] \to \hmodM[k]$ has a
left adjoint for every object $V$. Therefore, it is natural to wonder what we can say
about an algebra $A$ for which the category $\hmodM[A]$ of left $A$-modules is closed
(without necessarily being monoidal) in such a way that the forgetful functor
$\omega\colon \hmodM[A] \to \hmodM[k]$ is a closed functor, and only after analyse the case in which also the monoidal structure is lifted.

Despite being a perfectly natural and relevant question, the closed side of reconstruction theory seems to have been neglected in the literature so far. This paper is aimed at filling this apparent gap by studying algebras $A$ over the commutative ring $k$ for which the forgetful functor $\omega$ lifts the closed structure but not necessarily the monoidal one. We will mainly consider two different levels of strictness: skew- and (unital and associative) normal-closed structures.
In case the endofunctors $[X,-]$ on a closed category $\cat$ have left adjoint, which we then suggestively denote as $-\ot X$, then the closed structure is skew-closed if and only if the associated monoidal product $\ot$ is skew-monoidal (see \cite[\S2]{Szlachanyi-skew-monoidal}). Unital and associative normality of the closed structure then correspond respectively to (strong) unitality and associativity of the monoidal product. An important difference between monoidal and closed structures is that the normality conditions can no longer be expressed as {\em internal} properties in the category $\cat$, but they need to be expressed {\em externally} in the category of sets (or in any category over which $\cat$ is enriched, if one considers other than Set-enriched categories). This is a first aspect that makes the theory of lifting closed structures different from the usual Tannaka theory. Indeed, when lifting a monoidal product together with the associativity and unitality constraints along a faithful functor $\omega$, then the lifted structure is strong monoidal whenever the initial monoidal structure was. In the closed case, where normality is an external property, the normality of the lifted structure is an additional requirement.

It turns out that, for a given $k$-algebra $A$, lifting the skew-closed structure from $\hmodM[k]$ to $\hmodM[A]$ corresponds to the existence of algebra maps $\delta \colon A \to A \otimes A^{\mathrm{op}}$ and $\varepsilon \colon A \to k$ satisfying appropriate conditions. Namely,

\begin{theorem*}[\cref{prop:gabi_algebra_defining_properties}]
	Let $A$ be an algebra over a commutative ring $k$. Then the closed structure of $\hmodM[k]$ lifts to a skew-closed structure on $\hmodM[A]$ if and only if:
	\begin{enumerate}[leftmargin=0.8cm]
		\item there exists an algebra map $\counit \colon A \to k$ and
		\item there exists an algebra map $\delta \colon A \to A \tensor A\op$, $\delta(a) = a_+ \tensor a_-$ (summation understood), 
	\end{enumerate}
	such that for all $a \in A$
	\begin{gather*}
		a_+ \counit(a_-) = a \\
		a_+ a_- = \counit(a) 1_A \\
		a_{++} \tensor a_{-+} \tensor a_{--} a_{+-} = a_+ \tensor a_- \tensor 1
	\end{gather*}
	In this case, $A$ acts on $\hmodM[k](M, V)$ as $(a . f)(m) = a_+ f (a_- m)$.
\end{theorem*}
We call an algebra satisfying the equivalent conditions of the above Theorem a {\em gabi-algebra}\footnote{We prefer to keep the origin of the name gabi-algebra somewhat mysterious. Being an algebra with the same additional structure maps as a bialgebra, although satisfying other compatibility conditions (in particular, a different coassocativity condition), one could think of it as  a ``{\bf g}eneralized {\bf a}ssociative {\bf bi}algebra''. Another hint, however, might be \cite{Boehm-private}, where the above axioms were given for the first time.} and we show how a quite unexpected source of examples is provided by certain one-sided Hopf algebras in the sense of \cite{Green-Nichols-Taft}, i.e.\ bialgebras with a morphism which is just a one-sided convolution inverse of the identity. In particular, modules over a one-sided Hopf algebra are skew-closed and monoidal in such a way that the forgetful functor is closed and monoidal, but they do not form a closed monoidal category.
Furthermore, inspired by the existing structure on the module category of a gabi-algebra and by the prevailing examples, we also wonder when a gabi-algebra is a (one-sided) Hopf algebra and we provide necessary and sufficient conditions in \cref{sec:gabi-Hopf}.
Our most striking result in this direction is the following, which justifies the title of the paper in view of \cite{everybody}.

\begin{theorem*}[\cref{cor:urka}]
Let $A$ be an algebra. Then there is a bijective correspondence between normal gabi-algebra structures on $A$ and Hopf algebra structures on $A$.
\end{theorem*}

Concretely, our paper is organized as follows. In \cref{sec:closed} we recall the basics
of (skew-)closed and (skew-)monoidal categories and the duality between them.
\cref{sec:towardsGabi} is devoted to paving the way towards gabi-algebras. In \cref{ssec:EMcats} and \cref{ssec:closedmonsetting} we address the question of lifting a skew-closed structure in the monadic setting, from a skew-closed category $\cat$ to the Eilenberg-Moore category $\cat^T$ of algebras for a monad $T$ on $\cat$; they also contain our first main results, \cref{eq:sufficient_condition_for_lifting_closed_structure_to_EM} and \cref{prop:recasting_s_to_t}, providing necessary and sufficient conditions for the lifting. In \cref{ssec:gabi-algebras_in_set} we take advantage of the generality offered by the monadic perspective to fully answer the question in the set-theoretic setting and in \cref{ssec:gabi-algebras_in_kmod} we apply our machinery to tackle the linear setting, which is leading to the definition of gabi-algebras. With \cref{sec:gabialgebras_props_examples}, we finally introduce gabi-algebras, we study some elementary properties of them and we provide a few concrete examples. Finally, in \cref{sec:gabi-Hopf} we address the question of determining which additional conditions make a gabi-algebra into a Hopf algebra. There is also an appendix, \cref{appendix}, where we prove a lifting theorem for internal homs fairly more general than the corresponding part in \cref{eq:sufficient_condition_for_lifting_closed_structure_to_EM} and that represents, in fact, the core of the proof of the lifting property in \cref{eq:sufficient_condition_for_lifting_closed_structure_to_EM}.

As a matter of notation, $k$ denotes a fixed commutative ring (unless stated otherwise). If $A$ is a $k$-algebra, we denote by $1_A$ (or simply $1$) both its unit
and the $k$-linear map $k \to A$ giving the $k$-algebra structure. The multiplication of
$A$ might be denoted by $m$, $\cdot$ or simple juxtaposition. For a generic category
$\cat$, the external hom-set between two objects $X,Y$ is denoted by $\cat(X,Y)$, and the
identity morphism of an object $X$ may be denoted by $\id_X$ or $X$. The square brackets $[-,-]$ always denote an internal hom of a given closed structure and $\tensUnit$ a distinguished object (either the closed or the monoidal unit).

\section{Closed and monoidal categories}\label{sec:closed}
  In this section we review various flavours of closed categories and monoidal categories
  of different laxities.

\subsection{Left skew-closed categories}\label{ssec:skew-closed}
  First we turn to a suitably lax version of closed categories, following
  \cite{Street-skew_closed, UVZ_Eilenberg-Kelly_reloaded}, which has the unquestionable advantage of involving only conditions internal to $\cat$.
	
  \begin{definition}\label{def:skewclosed}
    A (\emph{left}) \emph{skew-closed category} is a tuple $(\cat, \tensUnit, i, j, \Gamma, [-,-])$, where 
    \begin{enumerate}[label={\itshape(\alph*)},leftmargin=0.8cm]
      \item
        $\tensUnit \in \cat$ is an object

      \item
      $[-,-] \colon \cat\op \times \cat \to \cat$ is a functor

      \item 
        $i \colon [\tensUnit, -] \To -$ is a natural transformation

      \item
        $j \colon \tensUnit \xrightarrow{\cdot\cdot} [-,-]$ is a dinatural transformation (see, e.g., \cite[Chapter IX, \S4]{Maclane}) 

      \item
        $\Gamma = \big\{\Gamma^X_{Y, Z} \colon [Y, Z] \to [[X, Y], [X, Z]]\big\}$ is a family of
        morphisms natural in the lower indices and dinatural in the upper index.
    \end{enumerate}
    These data are subject to the following axioms:
    \begin{equation*}
      \begin{tikzcd}
        \tensUnit \ar[r, "j_\tensUnit"] \ar[rd, equals]
        & {[\tensUnit, \tensUnit]} \ar[d, "i_\tensUnit"]
        \\ &
        \tensUnit
      \end{tikzcd}
      , \qquad
      \begin{tikzcd}
        {[X, Y]} \ar[r, "\Gamma^X_{X, Y}"] \ar[d, equals]
        & {[[X, X], [X, Y]]} \ar[d, "{[j_X, [X, Y]]}"]
        \\
        {[X, Y]}
        & {[\tensUnit, [X, Y]]} \ar[l, "{i_{[X, Y]}}"]
      \end{tikzcd}
    \end{equation*}

    \begin{equation*}
      \begin{tikzcd}
        \tensUnit \ar[r, "j_Y"] \ar[rd, swap, "{j_{[X, Y]}}"]
        & {[Y, Y]} \ar[d, "\Gamma^X_{Y, Y}"]
        \\ &
        {[[X, Y], [X, Y]]}
      \end{tikzcd}
      , \qquad
      \begin{tikzcd}
        {[X, Y]} \ar[r, "\Gamma^\tensUnit_{X, Y}"] \ar[rd, swap, "{[i_X, Y]}"]
        & {[[\tensUnit, X], [\tensUnit, Y]]} \ar[d, "{[[\tensUnit, X], i_Y]}"]
        \\ &
        {[[\tensUnit, X], Y]}
      \end{tikzcd}
    \end{equation*}

    \begin{equation*}
      \begin{tikzcd}[column sep=large]
        {[W, X]} \ar[r, "\Gamma^U_{W, X}"] \ar[dd, swap, "\Gamma^V_{W, X}"]
        & {[[U, W], [U, X]]} \ar[d, "{\Gamma^{[U, V]}_{[U, W], [U, X]}}"]
        \\ &
        {[[[U, V], [U, W]], [[U, V], [U, X]]]}
        \ar[d, "{[\Gamma^U_{V, W}, [[U, V], [U, X]]]}"]
        \\
        {[[V, W], [V, X]]}
        \ar[r, swap, "{[[V, W], \Gamma^U_{V, X}]}"]
        & {[[V, W], [[U, V], [U, X]]]}
      \end{tikzcd}
    \end{equation*}
    A skew-closed category is said to be 
    \begin{enumerate}[label=(N\arabic*),ref=(N\arabic*),leftmargin=1.2cm]
      \item\label{item:N1}
        \emph{left normal} if and only if
        \begin{align*}
          \hat\jmath_{X,Y}:\cat(X, Y) \to \cat(\tensUnit, [X, Y]), \quad
          f \mapsto [f, Y] \circ j_Y
        \end{align*}
        is a natural bijection;

        \item
          \emph{right normal} if and only if $i$ is a natural isomorphism;
		\item 
		  \emph{associative normal} if the canonical morphism
        \begin{align*}
          \hat\Gamma_{U,X,Y,Z}:\int^{V \in \cat} \cat(U, [V, Z]) \times \cat(X, [Y, V]) \to \cat(U, [X, [Y, Z]]),
        \end{align*}
defined component-wise as
\[
\hat\Gamma_{U,X,Y,Z}(f,g)\colon 
\xymatrix @C=40pt{
U \ar[r]^-f & [V,Z] \ar[r]^-{\Gamma^Y_{V,Z}} & [[Y,V],[Y,Z]] \ar[r]^-{[g,[Y,Z]]} & [X,[Y,Z]]
}
\]
for all $f\in \cat(U, [V, Z])$, $g\in \cat(X, [Y, V])$, is a bijection.
    \end{enumerate}
    A skew-closed category satisfying all three normality conditions will be called \emph{normal-closed}.
  \end{definition}

  \begin{remark}
    \begin{enumerate}[leftmargin=0.8cm]
    \item The adjective \emph{left} in front of the term \emph{skew-closed category} in \Cref{def:skewclosed} shall be interpreted as the adjective \emph{left} in \emph{left skew-monoidal category} (we recall this notion in \Cref{def:skewmonoidal}, for the convenience of the unaccustomed reader). In fact, as we will recall in \Cref{prop:bijection_skew_closed_monoidal_structures}, left skew-closed structures are adjoints to left skew-monoidal ones and vice-versa. Analogously, one could speak about \emph{right} skew-closed structures and these would be adjoints to \emph{right} skew-monoidal ones. The interested reader may refer to \cite[\S4]{UVZ_Eilenberg-Kelly_reloaded} for additional details. Henceforth, and as far as we are concerned, the left-hand side case would be enough and so we will often omit to specify it.
		
      \item
		The term \emph{closed category} is used in different ways in the literature, for a skew-closed category satisfying none, some or all of the normality conditions. Following \cite{UVZ_Eilenberg-Kelly_reloaded}, we agree that the term \emph{closed category} should be reserved for a skew-closed category that satisfies all three normality conditions, in light of the duality Theorem \ref{prop:bijection_skew_closed_monoidal_structures}.
However, in order to avoid confusion, we will call such a category ``normal-closed'' in this paper, and avoid to speak about ``closed category'' without any prefix.				
				
				\item In contrast to coherences of monoidal categories, left and associative normality for skew-closed categories are conditions \emph{external} to $\cat$, in the sense that they involve the fact that certain functions in the category of sets are bijections. They cannot be expressed by only using morphisms of $\cat$. This fact is the strongest feature of closed categories and it plays a crucial role in \cref{sec:gabi-Hopf}. Already at this stage it can be seen that a functor that preserves the closed structure, even in a strict way, will not necessarily preserve the normality conditions, because they are external.
				\qedhere
    \end{enumerate}
  \end{remark}

  Our main examples of skew-closed categories are the following ones.
		
  \begin{example}
    \label{ex:set_is_closed}
    The category $\Set$ of sets and functions is skew-closed.
    The inner hom functor is given by $[A, B] = \Set(A, B)$.
    The unit object is the unit object of the monoidal structure, i.e.\ a fixed
    one-element set $*$.
    The natural transformation $i_A \colon \Set(*, A) \to A$ is the isomorphism given by
    $i_A(f) = f(*)$, and the dinatural transformation $j_A \colon * \to \Set(A, A)$ picks
    out the identity, i.e.\ $j_A(*) = \id_A$.
    Finally, the transformation $\Gamma^A_{B, C}$ is given by post-composition, meaning
    \begin{equation*}
      \Gamma^A_{B, C} \colon \Set(B, C) \to \Set(\Set(A, B), \Set(A, C)),
      \quad f \mapsto (g \mapsto f \circ g)
      \ . \qedhere
    \end{equation*}
  \end{example}

  \begin{example}
    \label{ex:k-mod_is_closed}
    Let $k$ be a commutative ring with unit $1$.
    Its category of, say, left modules $\hmodM[k]$ is skew-closed.
    The inner hom is given by $[M, N] = \hmodM[k](M, N)$, on which $k$ acts as $(k .
    f)(m) = k f(m)$.
    The unit object is $k$.
    The natural transformation $i_M \colon \hmodM[k](k, M) \to M$ is the isomorphism
    given by $i_M(f) = f(1)$.
    The dinatural transformation $j_M$ is again `picking out the identity', i.e.\ it is
    the unique $k$-module map with $j_M(1) = \id_M$.
    Lastly, the transformation $\Gamma^M_{N, P}$ is again given by post-composition,
    $\Gamma^M_{N, P} (f) = f \circ -$.
  \end{example}
  
The categories in \Cref{ex:set_is_closed} and \Cref{ex:k-mod_is_closed} are both normal-closed, in fact. The interested reader may check this directly, but it follows more easily from the fact that they are closed monoidal (a notion that we will recall shortly in \cref{sec:skew_monoidal_cats}), in conjunction with the forthcoming \Cref{prop:bijection_skew_closed_monoidal_structures}. An example of a non normal-closed category is the following.
  
  \begin{example}
  	Let $k$ be a commutative ring and $R$ be a (not necessarily commutative) $k$-algebra. Consider the category $\hmodM[R]$ of left $R$-modules. For any $M,N\in \hmodM[R]$, define $[M,N]=\hmodM[k](M,N)$ with the following $R$-action.
    \[(r.f)(m) = r \cdot f(m)\]
  	for all $f\in \hmodM[k](M,N)$, $r\in R$ and $m\in M$. The unit object is given by the regular module $R$. The structure maps are given by
  	\begin{align*}
		i_M \colon & \hmodM[k](R,M)\to M, \qquad f \mapsto f(1_R), \\
  	j_M\colon & R\to \hmodM[k](M,M), \qquad r \mapsto \{m \mapsto r \cdot m\}
		\end{align*}
  	and $\Gamma$ is again given by post-composition. 
This is an associative normal but not unital normal closed structure. Again, the interested reader may check this directly, but it would be easier to observe that the following chain of bijections
\begin{align*}
\hmodM[R](.M,\hmodM[k](N,.P)) & \cong \hmodM[R](.M,\hmodM[R](.R. \ot N, .P)) \cong \hmodM[R]((.R. \ot N) \tensor_{R} .M,.P) \\
& \cong \hmodM[R](.M \ot N,.P)
\end{align*}
makes it clear that the given skew-closed structure $[-,-]$ corresponds, by adjunction, to the skew-monoidal structure $\ot$ with regular left $R$-action on the left-hand side tensor factor, which is an associative normal skew-monoidal structure, but not unital normal (see the forthcoming \Cref{ex:leftAmod}). 
  \end{example}

  Functors between skew-closed categories can be required to preserve the skew-closed
  structure in a coherent way.
  This leads to the following definition.

  \begin{definition}
    A functor $F \colon \cat \to \cat[D]$ between skew-closed  categories is
    \emph{closed} if there is a morphism $F_0 \colon \tensUnit_{\cat[D]} \to F\tensUnit_{\cat}$ and a
    natural transformation $F_2(X, Y) \colon F[X, Y]_{\cat} \to [FX, FY]_{\cat[D]}$
    satisfying
    \begin{gather*}
		\xymatrix @C=50pt{
		\tensUnit_{\cat[D]} \ar[r]^-{F_0} \ar[d]_-{j_{FX}} 
				& F \tensUnit_{\cat} \ar[d]^-{F j_X}  
				\\
        {[FX, FX]_{\cat[D]}}
        & F{[X, X]}_{\cat}
        \ar[l]^-{F_2(X, X)}
		}
			\\
    \xymatrix @C=50pt{
		FX \ar[r]^-{i_{FX}} \ar[d]_-{F i_X}
        & {[\tensUnit_{\cat[D]}, F X]}_{\cat[D]}
        \\
        {F[\tensUnit_{\cat}, X]_{\cat}} \ar[r]_-{F_2(\tensUnit_{\cat}, X)}
        &
        {[F\tensUnit_{\cat}, FX]_{\cat[D]}} \ar[u]_-{[F_0, FX]_{\cat[D]}}
				}
    \\
		\xymatrix @C=60pt{
		{F[X, Y]_{\cat}} \ar[r]^-{F\Gamma^Z_{X, Y}} \ar[d]_-{F_2(X, Y)}
        &
        {F[[Z, X]_{\cat}, [Z, Y]_{\cat}]_{\cat}} \ar[r]^-{F_2([Z, X], [Z,Y])}
        &
        {[F[Z, X]_{\cat}, F[Z, Y]_{\cat}]_{\cat[D]}} \ar[d]^-{[\id, F_2(Z, Y)]}
        \\
        {[FX, FY]_{\cat[D]}} \ar[r]_-{\Gamma^{FZ}_{FX, FY}}
        &
        {[[FZ, FX]_{\cat[D]}, [FZ, FY]_{\cat[D]}]_{\cat[D]}} \ar[r]_-{[F_2(Z, X), \id]}
        &
        {[F[Z, X]_{\cat}, [FZ, FY]_{\cat[D]}]_{\cat[D]}}
				}
    \end{gather*}
    If all of these are isomorphisms (identities), then $F$ is called
    \emph{strong (strict) closed}.
  \end{definition}

  \begin{example}
    The identity functor on a skew-closed category is strict closed.
    The composition of closed functors is closed: if $F, G$ are composable closed functors, then $FG$ is closed with
    $(FG)_0 = FG_0 \circ F_0$ and $(FG)_2(X, Y) = F_2(GX, GY) \circ FG_2(X, Y)$.
  \end{example}

\subsection{Skew-monoidal categories}
\label{sec:skew_monoidal_cats}
  The notion of a skew-monoidal category is dual to that of a skew-closed category, in a sense to be
  made precise in \Cref{prop:bijection_skew_closed_monoidal_structures}.

  \begin{definition}\label{def:skewmonoidal}
    Following \cite{Szlachanyi-skew-monoidal}, a (\emph{left}) \emph{skew-monoidal category} is a
    tuple $(\cat, \tensor, \tensUnit, \alpha, \lambda, \rho)$ where $\cat$ is a category,
    $\tensUnit$ is an object of $\cat$, $\tensor \colon \cat \times \cat \to \cat$ is a
    functor, and
		\begin{equation}\label{eq:monoidal}
    \begin{gathered}
      \alpha_{X, Y, Z} \colon (X \tensor Y) \tensor Z \to X \tensor (Y \tensor Z)
      \ , \\
      \lambda_X \colon \tensUnit \tensor X \to X
      \ , \qquad
      \rho_X \colon X \to X \tensor \tensUnit
      \ ,
    \end{gathered}
		\end{equation}
    are natural transformations subject to the following axioms:
    \begin{gather*}
		\xymatrix{
		& (W \ot (X \ot Y)) \ot Z \ar[rd]^-{\alpha}
        &
        \\
        ((W \ot X) \ot Y) \ot Z \ar[ru]^-{\alpha \ot Z} \ar[d]_-{\alpha}
        & &
        W \ot ((X \ot Y) \ot Z) \ar[d]^-{W \ot \alpha}
        \\
        (W \ot X) \ot (Y \ot Z) \ar[rr]^-{\alpha}
        & &
        W \ot (X \ot (Y \ot Z))
				}
    \\
    \xymatrix{
		X \ot Y \ar[r]^-{\rho} \ar[rd]_-{X \ot \rho}
        & (X \ot Y) \ot \tensUnit \ar[d]^-{\alpha}
        \\ &
        X \ot (Y \ot \tensUnit)
		}
      \qquad
		\xymatrix{
		(\tensUnit \ot X) \ot Y \ar[r]^-{\alpha} \ar[d]_-{\lambda \ot Y}
        & \tensUnit \ot (X \ot Y) \ar[ld]^-{\lambda}
        \\
        X \ot Y
		}
    \\
		\xymatrix @C=35pt{
		X \ot Y \ar[r]^-{\rho \ot Y} \ar@{=}[d]
        & (X \ot \tensUnit) \ot Y \ar[d]^-{\alpha}
        \\
        X \ot Y & X \ot (\tensUnit \ot Y) \ar[l]^-{X \ot \lambda}
		}
      \qquad
		\xymatrix{
		\tensUnit \ar[d]_-{\rho} \ar@{=}[dr]
        \\
        \tensUnit \ot \tensUnit \ar[r]_-{\lambda}
        & \tensUnit
		}
    \end{gather*}
    A skew-monoidal category is
    \begin{enumerate}[leftmargin=0.8cm]
      \item
        \emph{left normal} if and only if $\lambda$ is an isomorphism,
      \item
        \emph{right normal} if and only if $\rho$ is an isomorphism,
      \item
        \emph{associative normal} if and only if $\alpha$ is an isomorphism.
    \end{enumerate}  
    A \emph{monoidal category} is a skew-monoidal category satisfying
    all three normality conditions.
  \end{definition}
	
	\begin{remark}
	A \emph{right} skew-monoidal category would be defined similarly, but with all structure morphisms \eqref{eq:monoidal} reversed. Since we will be interested mainly in left skew-monoidal categories, we will often omit to specify it.
	\end{remark}
	
	\begin{example}\label{ex:leftAmod}
	Let $k$ be a commutative ring with unit and let $R$ be a $k$-algebra. The category $\hmodM[R]$ of left $R$-modules is an associative normal skew-monoidal category with respect to the tensor product $\otimes$ over $k$ and the natural transformations
	\begin{equation}\label{eq:assunitors}
	\begin{gathered}
	\alpha_{M,N,P} \colon (M \otimes N) \otimes P \xrightarrow{\cong} M \otimes (N \otimes P), \qquad (m \otimes n) \otimes p \mapsto m \otimes (n \otimes p), \\ 
	\lambda_M \colon R \otimes M \to M, \qquad r \otimes m \mapsto r \cdot m, \\
	\rho_M \colon M \to M \otimes R, \qquad m \mapsto m \otimes 1_R.
	\end{gathered} \qedhere
	\end{equation}
	\end{example}

\subsection{Skew-closed skew-monoidal categories}\label{ssec:skewclosed}
  Let now $\cat$ be a skew-closed category, with the usual notation, and assume that
  for each $X$ in $\cat$ there is an adjunction $L_X \dashv [X, -]$.
  Denote by $\coev^X$ and $\ev^X$ the unit and counit of each adjunction.
  Since $[-,-]$ is a bifunctor, there is a unique way to assign to each arrow $f \colon X \to Y$ and each object $Z$ of $\cat$ an arrow $L_f \colon L_XZ \to L_YZ$ of $\cat$ so that $(X, Y) \mapsto L_X Y$ becomes a bifunctor for which the bijection $\cat(L_XY,Z) \cong \cat(Y,[X,Z])$ is natural in all three variables (cf.\ the symmetric version of \cite[Theorem IV.7.3]{Maclane}).
  The action of $L$ on morphisms $f \colon X \to Y$ in the lower index is given by
  \begin{align*}
    L_f Z \colon L_X Z
    \xrightarrow{L_X \coev^Y_Z} L_X [Y, L_Y Z]
    \xrightarrow{L_X [f, L_Y Z]} L_X [X, L_Y Z]
    \xrightarrow{\ev^X_{L_Y Z}} L_Y Z
    \ .
  \end{align*}
  For later use we record the following property of the adjunctions above (see \cite[\S IX.4, page 216]{Maclane}).

  \begin{lemma}
    \label{prop:ev_coev_dinatural}
    The evaluations and coevaluations $\ev^X$ and $\coev^X$ are dinatural in $X$.
  \end{lemma}

  \begin{proof}
    We show it for the evaluation, the coevaluation being completely analogous.
    One computes for $f \colon X \to Y$
    \begin{align*}
      \ev^Y_Z
      \circ L_f [Y, Z]
      &=
      \ev^Y_Z
      \circ \ev^X_{L_Y [Y, Z]}
      \circ L_X [f, L_Y [Y, Z]]
      \circ L_X \coev^Y_{[Y, Z]}
      \\ &\oversetEq[(*)]
      \ev^X_{Z}
      \circ L_X [f, Z]
      \circ L_X [Y, \ev^Y_Z]
      \circ L_X \coev^Y_{[Y, Z]}
      \\ &\oversetEq[(\star)]
      \ev^X_{Z}
      \circ L_X [f, Z]
    \end{align*}
    using $(*)$ naturality of $\ev^X$, and $(\star)$ one of the triangles of the adjunction
    $L_X \dashv [X, -]$.
  \end{proof}

  Transporting the closed structure through the
  adjunction imparts on $\cat$ the structure of a skew-monoidal category
  \cite[Proposition 18]{Street-skew_closed}.
  Let us describe here how the structure is obtained.
  We have the natural transformation
  \begin{align}
    p_{X, Y, Z} \colon
    [L_Y X, Z] \xrightarrow{\Gamma^{Y}_{L_Y X, Z}}
    [[Y, L_Y X], [Y, Z]] \xrightarrow{ [\coev^Y_X, [Y, Z]] }
    [X, [Y, Z]]
    \ .
    \label{eq:associativity_between_inner_hom_and_tensor}
  \end{align}
  From this, one constructs the natural transformation
  \begin{align}
    \cat(L_{L_Z Y} X, V) 
    \xrightarrow{\sim} \cat(X, [L_Z Y, V])
    \xrightarrow{p_{Y, Z, V} \circ -} \cat(X, [Y, [Z, V]])
    \xrightarrow{\sim} \cat(L_Z L_Y X, V)
    \ .
    \label{eq:nat_transformation-associator}
  \end{align}
  Such natural transformations are in bijection with $\cat(L_Z L_Y X, L_{L_Z Y} X)$
  by the Yoneda lemma,  and this gives a \emph{skew associator} $\alpha_{X, Y, Z} \colon L_Z L_Y X 
  \to L_{L_Z Y} X$ as the morphism corresponding to
  \eqref{eq:nat_transformation-associator}.

  The two \emph{skew unitors} are obtained from the adjunction as follows.
  Firstly, $\lambda_X \colon L_X \tensUnit \to X$ corresponds simply to $j_X \colon \tensUnit \to [X,
  X]$.
  For the other, consider the natural transformation
  \begin{align}
    \cat(L_\tensUnit X, Y) 
    \xrightarrow{\sim} \cat(X, [\tensUnit, Y]) 
    \xrightarrow{i_Y \circ -} \cat(X, Y)
    \ ,
    \label{eq:unitors_from_adjunction}
  \end{align}
  which by the Yoneda lemma gives $\rho_X \colon X \to L_\tensUnit X$.

  Defining $\otimes$ by setting $- \otimes X \coloneqq L_X$ concludes the construction of the skew-monoidal structure $(\cat, \tensor, \tensUnit, \alpha,
  \lambda, \rho)$, and we repeat here the following theorem.

  \begin{theorem}[{\cite[Theorems 2.10 and 3.8]{UVZ_Eilenberg-Kelly_reloaded}}]
    \label{prop:bijection_skew_closed_monoidal_structures}
    Let $\cat$ be a category with a distinguished object $\tensUnit$, and functors $-
    \tensor - \colon \cat \times \cat \to \cat$ and $[-,-] \colon \cat\op \times \cat
    \to \cat$.
    Assume that there are adjunctions $- \tensor X \dashv [X, -]$, natural in $X \in
    \cat$.
    Then left skew-monoidal structures $(\alpha, \lambda, \rho)$ on $(\cat, \tensor,
    \tensUnit)$ are in bijection with left skew-closed structures $(\Gamma, j, i)$ on
    $(\cat, [-,-], \tensUnit)$.

    Moreover, the left skew-monoidal structure is left/right/associative normal if and only if the
    left skew-closed structure is left/right/associative normal.
    The left skew-monoidal structure is associative normal if and only if $p$ from
    \eqref{eq:associativity_between_inner_hom_and_tensor} is a natural isomorphism.
  \end{theorem}

  \begin{remark}
    \label{rem:closed_monoidal_category}
    \begin{enumerate}[label=(\arabic*),ref=(\arabic*),leftmargin=0.8cm]
\item Remark that $p$ from \eqref{eq:associativity_between_inner_hom_and_tensor} being a natural isomorphism can be understood as $L_X$ being a left inverse for $[X,-]$, as functors between $\cat$-enriched categories.
				
			\item\label{item:closed3} For the sake of clearness, let us provide explicit formulae for the monoidal constraints obtained from the closed ones under the correspondence of \Cref{prop:bijection_skew_closed_monoidal_structures}.
			The associativity constraint $\alpha_{X,Y,Z}$ in $\cat((X \otimes Y) \otimes Z, X \otimes (Y \otimes Z))$ is given by
			\begin{align*} 
			\alpha_{X,Y,Z} = \ev^Z_{X \otimes (Y \otimes Z)} & \circ \big(\ev^Y_{[Z,X\otimes (Y \otimes Z)]} \otimes Z\big) \circ \big(([\coev^{Z}_Y,[Z,X \otimes (Y \otimes Z)]] \otimes Y ) \otimes Z\big) \\
			& \circ \big((\Gamma^Z_{Y \otimes Z,X \otimes (Y \otimes Z)} \otimes Y) \otimes Z\big) \circ \big((\coev_X^{Y \otimes Z} \otimes Y) \otimes Z\big)
			\end{align*}
			and it is the unique morphism such that
			\begin{equation}\label{eq:alphaunique}
			\begin{aligned}
			[Y,[Z,\alpha_{X,Y,Z}]] & \circ [Y, \coev_{X \otimes Y}^Z ]\circ \coev_X^Y \\
			& = [\coev_Y^Z,[Z,X \otimes (Y \otimes Z)]] \circ \Gamma_{Y \otimes Z,X\otimes(Y \otimes Z)}^Z \circ \coev_{X}^{Y \otimes Z};
			\end{aligned}
			\end{equation}
			the left-unit constraint $\lambda_X\in\cat(\tensUnit \otimes X,X)$ is given by $\lambda_X = \ev_X^X \circ (j_X \otimes X)$ and it is the unique morphism such that
			\[[X, \lambda_X] \circ \coev^X_\tensUnit = j_X;\]
			the right-unit constraint $\rho_X \in \cat(X,X \otimes \tensUnit)$ is given by $\rho_X = i_{X \otimes \tensUnit} \circ \coev^\tensUnit_X$. \qedhere
    \end{enumerate}
  \end{remark}

  \begin{definition}\label{def:left_skew_closed_skew_monoidal}
    A (left) skew-closed and (left) skew-monoidal category is called (\emph{left}) \emph{skew-closed
    skew-monoidal} if and only if the skew-closed and the skew-monoidal structure are dual to one
    another in the sense of
    \Cref{prop:bijection_skew_closed_monoidal_structures}.
    If either, and hence both, structure satisfies all three normality conditions,
    this situation is termed a \emph{right closed monoidal category} in literature
    (remark however that the adjective \emph{right} here has a different meaning than the {\em left} before: it indicates that the left adjoint to $[X, -]$ is given by tensoring
    on the {\em right}).
    A monoidal category is \emph{closed} if all functors $X \tensor -$ and $- \tensor
    X$ have right adjoints (leading to two, \textit{a priori} distinct, closed structures).
  \end{definition}

  \begin{example}
    The category $\Set$ is closed monoidal, the monoidal product being given by the
    categorical product. The closed structure (both left and right) is the one in \Cref{ex:set_is_closed}.
  \end{example}

  \begin{example}
    \label{ex:modules_over_comm_ring_mon_closed}
    Let $R$ be a ring, and consider the category $\hmodM[R][R]$ of $R$-bimodules.
    This is a monoidal category under $\tensor_R$, and it is in fact closed monoidal.
    Indeed, recall that for bimodules over rings $R, S, T$, we have the tensor-hom
    adjunctions
    \begin{align}
      \hmodM[S][T](N, \hmodM[R][](M, P))
      \cong \hmodM[R][T](M \tensor_S N, P)
      \cong \hmodM[R][S](M, \hmodM[][T](N, P))
      \label{eq:tensor_hom_adjunction}
      \ ,
    \end{align}
    where the bimodule structures on the hom spaces are of course $(sft)(m) = f(ms)t$ and
    $(rfs)(n) = rf(sn)$.
    Specializing to $S = T = R$, we obtain the left and right closed monoidal structures
    associated to the adjunctions
    \begin{align*}
      M \otimes_R - \dashv \hmodM[R](M, -)
      \qandq
      - \otimes_R M \dashv \hmodM[][R](M, -)
    \end{align*}
    of endofunctors on $\hmodM[R][R]$.
    Note that the two internal homs are in general not isomorphic.
    For later use, let us record that the evaluations and coevaluations for the left
    closed monoidal structure are
    \begin{gather*}
      \ev^{M}_N \colon M \otimes_R \hmodM[R](M, N) \to N,\quad 
      m \otimes_R f \mapsto f(m), \qquad \text{and} \\
      \coev^{M}_N \colon N \to \hmodM[R](M, M \otimes_R N),\quad
      n \mapsto (m \mapsto m \otimes_R n)
      \ .
    \end{gather*}
    The corresponding right versions are completely analogous.

    If $R = k$ is commutative, then the category of e.g.\ left modules $\hmodM[k]$ is
    closed monoidal as well.
    This is facilitated by first embedding $\hmodM[k]$ into $\hmodM[k][k]$ (as usual, a
    right action on $M \in \hmodM[k]$ is defined by $mr = rm$), then performing the
    relevant operations there, and finally forgetting back to $\hmodM[k]$.
    Note that in this case, $M \otimes N \cong N \otimes M$ canonically, and hence
    the left internal hom and the right internal hom agree.
    We obtain the closed structure from \Cref{ex:k-mod_is_closed}.
    The same can be done for right $k$-modules.
  \end{example}

\section{Towards Gabi-algebras}\label{sec:towardsGabi}
  In this section, $\cat$ is at least a skew-closed category.
  Let $T \in \End(\cat)$ be a monad on $\cat$.

\subsection{Lifting to Eilenberg-Moore categories}\label{ssec:EMcats}
  We wish to lift the skew-closed structure of $\cat$ to the Eilenberg-Moore category
  $\cat^T$.
  Thus we need to find a functor $[-,-]^T \colon (\cat^T)\op \times \cat^T \to \cat^T$,
  and natural transformations $i^T, j^T, \Gamma^T$, forming a skew-closed structure
  on $\cat^T$ and such that the forgetful functor $U_T \colon \cat^T \to \cat$ is
  strictly closed.
  Then, $i^T = i$, $j^T = j$, $\Gamma^T = \Gamma$.
  Thus, right normality of $\cat$ is immediately transferred to $\cat^T$, while this does
  not necessarily have to be the case for the other two normality conditions.

  We invoke the theorem for mixed liftings, \Cref{prop:mixed_liftings}, to obtain the following.

  \begin{lemma}[\Cref{prop:mixed_liftings_corollary}]
    \label{prop:liftings_of_inner_hom}
    Let $\cat$ be a category, $[-,-] \colon \cat\op \times \cat \to \cat$ a functor, and
    $(T, m, u)$ be a monad on $\cat$.
    The liftings of $[-,-]$ to $\cat^T$ are in bijective correspondence with natural
    transformations $s_{X, Y} \colon T[TX, Y] \to [X, TY]$ satisfying
    \begin{align}
      \label{eq:lifting_closed_unitality}
      s_{X, Y} \circ u_{[TX, Y]} 
      & = [u_X, u_Y], \\
      \label{eq:lifting_closed_multiplicativity}
      s_{X, Y} \circ m_{[TX, Y]} & = [X, m_Y] \circ s_{X, TY} \circ T s_{TX, Y} \circ
      T^2 [m_X, Y]
    \end{align}
    for all $X, Y \in \cat$.
  \end{lemma}

  We call the conditions
  \cref{eq:lifting_closed_unitality,eq:lifting_closed_multiplicativity} on $s$
  \emph{unitality} and \emph{multiplicativity}, respectively.

  \smallskip

  Let us make explicit the bijection alluded to in
  \Cref{prop:liftings_of_inner_hom}.
  Let first $[-,-]^T$ be a lifting of $[-,-]$.
  For all $(M, \mu_M), (N, \mu_N) \in \cat^T$, we then in particular have that the
  object $[M, N]$ is equipped with a some action of $T$ that we denote by $\mu_M \star \mu_N \colon T[M, N] \to 
  [M, N]$.
  We can thus define a natural transformation as above via
  \begin{align*}
    T [TX, Y] \xrightarrow{T[TX, u_Y]} T[TX, TY] \xrightarrow{\mu_X \star \mu_Y} [TX, TY]
    \xrightarrow{[u_X, Y]} [X, TY]
    \ ,
  \end{align*}
  where $X, Y \in \cat$.
  Conversely, given a natural transformation $s_{X, Y} \colon T[TX, Y] \to [X, TY]$, we
  can define an action $\mu \star \sigma$ via
  \begin{align*}
    T[M, N] \xrightarrow{T[\mu_M, N]} T[TM, N] \xrightarrow{s_{M, N}} [M, TN]
    \xrightarrow{[M, \mu_N]} [M, N]
    \ ,
  \end{align*}
  where now of course $(M, \mu_M), (N, \mu_N) \in \cat^T$.

  \begin{example}
    \label{ex:Hopf_monad_1}
    Let $T$ be a Hopf monad on the left closed monoidal category $\cat$ (in the sense of \cite[\S3.1]{BLV-hopf_monads}).
    Then $T$ admits a \emph{left antipode}, which is a natural transformation $s_{X, Y}
    \colon T[TX, Y] \to [X, TY]$ satisfying some conditions only valid in the closed
    monoidal setting \cite[\S3.3]{BLV-hopf_monads}.
    However, by \cite[Proposition 3.8]{BLV-hopf_monads}, the antipode is an example of a
    natural transformation as in \Cref{prop:liftings_of_inner_hom}, which therefore may
    be seen as a generalization of the antipode to the closed but non-monoidal setting.
  \end{example}

  \smallskip

	\begin{remark}\label{rem:gamma}
	The interested reader may check that there is a bijective correspondence between natural transformations $s_{X, Y} \colon T[TX, Y] \to [X, TY]$ for $X,Y$ in $\cat$ and natural transformations $\gamma^M_Y \colon T[M,Y] \to [M,TY]$ for $Y$ in $\cat$ and $(M,\mu_M)$ in $\cat^T$. Indeed, given $s$ we can consider
	\[\gamma(s)_Y^M \colon T[M,Y] \xrightarrow{T[\mu_M,Y]} T[TM,Y] \xrightarrow{s_{M,Y}} [M,TY].\]
	In the opposite direction, given $\gamma$ we can consider
	\[s(\gamma)_{X,Y} \colon T[TX,Y] \xrightarrow{\gamma_Y^{TX}} [TX,TY] \xrightarrow{[u_X,TY]} [X,TY].\]
	These two constructions are well-defined and each others inverses. Furthermore, $s$ satisfies \eqref{eq:lifting_closed_unitality} if and only if $\gamma$ satisfies 
	\begin{equation}\label{eq:gamma1}
	\gamma_Y^M \circ u_{[M,Y]} = [M,u_Y]
	\end{equation}
	and $s$ satisfies \eqref{eq:lifting_closed_multiplicativity} if and only if $\gamma$ satisfies
	\begin{equation}\label{eq:gamma2}
	[M,m_Y] \circ \gamma^M_{TY} \circ T(\gamma^M_Y) = \gamma^M_Y \circ m_{[M,Y]}.
	\end{equation}
	The $T$-algebra structure on $[M,N]$ by means of $\gamma$ is given by
	\[T[M,N] \xrightarrow{\gamma^M_N} [M,TN] \xrightarrow{[M,\mu_N]} [M,N]\]
	for all $(M,\mu_M)$, $(N,\mu_N)$ in $\cat^T$.
	
	The fact that for any $(M,\mu_M)$ there exists $\gamma^M_Y$ natural in $Y$ and satisfying \eqref{eq:gamma1} and \eqref{eq:gamma2} is equivalent to require that there exists a functor $[(M,\mu_M),-]^T \colon \cat^T \to \cat^T$ such that $\omega[(M,\mu_M),-]^T = [M,\omega(-)]$. Naturality of $\gamma$ in $M$ is equivalent to the fact that the assignment $(\cat^T)^{\mathrm{op}} \to \mathsf{Funct}(\cat^T,\cat^T), (M,\mu_M) \mapsto [(M,\mu_M),-]^T,$ is functorial. See \cite[\S2.1]{Lack-Street-formal}, or \cite[Theorems 2.27 and 2.30]{Gabi-book} for a more elementary approach.
	\end{remark}

  In order to obtain a lifting of the entire skew-closed structure, one now simply needs to
  require two things:
  the existence of some $T$-algebra structure $T \tensUnit \to \tensUnit$ on $\tensUnit$
  and that the structural morphisms $i, j, \Gamma$ of $\cat$ are intertwiners of
  the $T$-algebra structures granted by the lifting $[-,-]_T$ of the bifunctor $[-,-]$.
The following theorem collects necessary and sufficient conditions for the skew-closed
  structure to lift.
	
  \begin{theorem}
    \label{eq:sufficient_condition_for_lifting_closed_structure_to_EM}
    Let $T$ be a monad on the left skew-closed category $\cat$. Then $\cat^T$ is left skew-closed such that the forgetful functor to $\cat$ is
    strictly closed if and only if:
    \begin{enumerate}[leftmargin=0.8cm]
      \item
        there is a morphism $\mu_\tensUnit \colon T \tensUnit \to \tensUnit$ such that
        $(\tensUnit, \mu_\tensUnit) \in \cat^T$,

      \item
        $T$ lifts $[-,-]$ via a natural transformation $s \colon T[T-,{\sim}] \To [-, T{\sim}]$
        as in \cref{prop:liftings_of_inner_hom},
    \end{enumerate}
		and they satisfy
		\begin{gather}
          T i_{X} = i_{T X} \circ s_{\tensUnit, X} \circ T[\mu_\tensUnit, X]
          , \quad \forall\, X \in \cat, \label{eq:siota} \\
          j_M \circ \mu_\tensUnit = [M, \mu_M \circ m_M] \circ s_{M, TM} \circ T j_{T M}
          , \quad \forall\, (M, \mu_M) \in \cat^T, \label{eq:sj}
        \end{gather}
        and for all $X, Y \in \cat$ and $(P, \mu_P) \in \cat^T$
        \begin{equation} \label{eq:sGamma}
          \Gamma^P_{X, TY} \circ s_{X, Y}
          = [[\mu_P, X], s_{P, Y}] \circ s_{[TP, X], [TP, Y]} 
          \circ T[s_{P, X}, [\mu_P, Y]] \circ T \Gamma^P_{TX, Y}
          \ .
        \end{equation}
  \end{theorem}
	
	\begin{proof}
	Let us prove the necessity part. Suppose that $i_M$ is a morphism of $T$-algebras for any $(M,\mu_M)$ in $\cat^T$. Then, given any $X$ in $\cat$, in the following diagram
	\[
	\xymatrix @C=40pt @R=40pt{
	T[T\tensUnit,X] \ar@{}[dr]|-{(a)} \ar[d]_-{T[T\tensUnit,u_X]} & T[\tensUnit,X] \ar@{}[dr]|-{(b)} \ar[l]_-{T[\mu_\tensUnit,X]} \ar[d]|-{T[\tensUnit,u_X]} \ar[r]^-{Ti_{X}} & {TX} \ar[d]|-{Tu_{X}} \ar@/^5ex/[dd]^{\id_{{TX}}} \\
	T[T\tensUnit,{TX}] \ar[d]_-{s_{\tensUnit,{TX}}}\ar@{}[dr]|-{(c)}  & T[\tensUnit,{TX}] \ar[d]|-{\mu_{[\tensUnit,{TX}]}} \ar@{}[dr]|-{(d)} \ar[l]^-{T[\mu_\tensUnit,{TX}]} \ar[r]_-{Ti_{TX}} & T^2X \ar[d]|-{\mu_{TX}} \\
	[\tensUnit,T^2X] \ar[r]_-{[\tensUnit,\mu_{TX}]} & [\tensUnit,{TX}] \ar[r]_-{i_{TX}} & {TX}
	}
	\]
	$(a)$ commutes by bifunctoriality of $[-,-]$, $(b)$ commutes by naturality of $i$, $(c)$ commutes by definition of $\mu_{[\tensUnit,M]}$ and $(d)$ commutes by hypothesis. The right-most relation follows from the fact that $\mu_{TX} = m_X$. Since, moreover, $s$ is natural in both entries, $s_{\tensUnit,TX} \circ T[T\tensUnit,u_X] = [\tensUnit,Tu_X] \circ s_{\tensUnit,X}$, whence
	\[Ti_X = i_{TX} \circ [\tensUnit,\mu_{TX}\circ Tu_X] \circ s_{\tensUnit,X} \circ T[\mu_\tensUnit,X] = i_{TX} \circ s_{\tensUnit,X} \circ T[\mu_\tensUnit,X]\]

Now, suppose that $j_M$ is a morphism of $T$-algebras for every $(M,\mu_M)$ in $\cat^T$. This entails that the following diagram commutes for every $(M,\mu_M)$ in $\cat^T$
\[
\xymatrix @C=40pt @R=40pt{
 & T[TM,TM] \ar[dr]^-{T[TM,\mu_M]} \ar[rr]^-{s_{M,TM}} \ar@{}[d]|-{(a)} & & [M,T^2M] \ar[ddl]^-{[M,T\mu_M]}  \ar@{}[dl]|-{(b)} \\
T\tensUnit  \ar@{}[dr]|-{(c)} \ar[r]_-{Tj_M} \ar[d]_-{\mu_\tensUnit} \ar[ur]^-{Tj_{TM}} & T[M,M] \ar@{}[dr]|-{(d)} \ar[d]^-{\mu_{[M,M]}} \ar[r]_-{T[\mu_M,M]} & T[TM,M] \ar[d]|-{s_{M,M}} & \\
\tensUnit \ar[r]_-{j_M} & [M,M] & [M,TM] \ar[l]^-{[M,\mu_M]} &
}
\]
since $(a)$ commutes by dinaturality of $j$, $(b)$ commutes by naturality of $s$, $(c)$ commutes by hypothesis and $(d)$ commutes by definition of $\mu_{[M,M]}$. That is,
\[j_M \circ \mu_\tensUnit = [M,\mu_M\circ T\mu_M] \circ s_{M,TM} \circ Tj_{TM} = [M,\mu_M\circ m_M] \circ s_{M,TM} \circ Tj_{TM}.\]

Finally, suppose that $\Gamma_{M,N}^P$ is a morphism of $T$-algebras for all $(M,\mu_M)$, $(N,\mu_N)$, $(P,\mu_P)$ in $\cat^T$. This entails that the following diagram commutes for all $(M,\mu_M)$, $(N,\mu_N)$, $(P,\mu_P)$ in $\cat^T$
{\footnotesize
\[
\xymatrix @C=40pt @R=35pt{
T[TM,N] \ar@{}[dr]|-{(a)} \ar@/_1.3cm/[ddd]|-{s_{M,N}} \ar[r]^-{T\Gamma_{TM,N}^P} & T[[P,TM],[P,N]] \ar@{}[dr]|-{(b)} \ar[r]^-{T[s_{P,M},[P,N]]} & T[T[TP,M],[P,N]] \ar[d]|-{T[T[\mu_P,M],[P,N]]} \ar@/^3ex/[dr]|-{T[T[TP,M],[\mu_P,N]]} &  \\
T[M,N] \ar@{}[dr]|-{(c)} \ar[r]^-{T\Gamma_{M,N}^P} \ar[d]|-{\mu_{[M,N]}} \ar[u]|-{T[\mu_M,N]} & T[[P,M],[P,N]]  \ar@{}[dr]|-{(d)} \ar[u]|-{T[[P,\mu_M],[P,N]]} \ar[d]|-{\mu_{[[P,M],[P,N]]}} \ar[r]^-{T[\mu_{[P,M]},[P,N]]} & T[T[P,M],[P,N]] \ar@{}[dr]|-{(g)} \ar[d]|-{s_{[P,M],[P,N]}} & T[T[TP,M],[TP,N]] \ar[d]|-{s_{[TP,M],[TP,N]}} \\
[M,N] \ar@{}[dr]|-{(e)} \ar[r]_-{\Gamma_{M,N}^P} & [[P,M],[P,N]]  \ar@{}[dr]|-{(f)} & [[P,M],T[P,N]] \ar[l]^-{[[P,M],\mu_{[P,N]}]} \ar[d]|-{[[P,M],T[\mu_P,N]]} & [[TP,M],T[TP,N]] \ar@/^3ex/[dl]|-{[[\mu_P,M],T[TP,N]]} \\
[M,TN] \ar[r]_-{\Gamma_{M,TN}^P} \ar[u]|-{[M,\mu_N]} & [[P,M],[P,TN]] \ar[u]|-{[[P,M],[P,\mu_N]]} & [[P,M],T[TP,N]] \ar[l]^-{[[P,M],s_{P,N}]}
}
\]
}%
since $(a)$ and $(e)$ commute by naturality of $\Gamma$, the left-most bended diagram, $(b)$, $(d)$ and $(f)$ commute by definition of $\mu_{[M,N]}$, $(c)$ commutes by hypothesis, and $(g)$ commutes by naturality of $s$. As a consequence,
\begin{align*}
& \Gamma_{M,TY}^P \circ s_{M,Y} \circ T[\mu_M,Y] = \\
& = [[P,M],[P,\mu_{TY}]] \circ \Gamma_{M,T^2Y}^P \circ s_{M,TY} \circ T[\mu_M,TY] \circ T[M,u_Y] \\
& \begin{aligned} ~\stackrel{(*)}{=}~ & [[P,M],[P,\mu_{TY}]] \circ [[\mu_P,M],s_{P,TY}] \circ s_{[TP,M],[TP,TY]} \circ T[s_{P,M},[\mu_P,TY]] ~\circ \\ & \circ ~ T\Gamma_{TM,TY}^P \circ T[\mu_M,TY] \circ T[M,u_Y] \end{aligned} \\
& = [[\mu_P,M],s_{P,Y}] \circ s_{[TP,M],[TP,Y]} \circ T[s_{P,M},[\mu_P,Y]] \circ T\Gamma_{TM,Y}^P \circ T[\mu_M,Y]
\end{align*}
where $(*)$ is the commutativity of the diagram above with $N = TY$, and therefore
\begin{align*}
& \Gamma_{X,TY}^P \circ s_{X,Y} = \\
& = [[P,u_X],[P,TY]] \circ \Gamma_{TX,TY}^P \circ s_{TX,Y} \circ T[\mu_{TX},Y] \\
& \begin{aligned} ~=~ & [[P,u_X],[P,TY]] \circ [[\mu_P,TX],s_{P,Y}] \circ s_{[TP,TX],[TP,Y]}~\circ \\ & \circ ~ T[s_{P,TX},[\mu_P,Y]] \circ T\Gamma_{T^2X,Y}^P \circ T[\mu_{TX},Y] \end{aligned} \\
& = [[\mu_P,X],s_{P,Y}] \circ s_{[TP,X],[TP,Y]} \circ T[s_{P,X},[\mu_P,Y]] \circ T\Gamma_{TX,Y}^P
\end{align*}
for all $X,Y$ in $\cat$, by naturality of the maps involved and because $\mu_{TX} = m_X$.

To conclude, a close inspection of the diagrams above shall convince the reader that the three conditions are also sufficient.
	\end{proof}

	\begin{remark}\label{rem:gamma2}
	Continuing \cref{rem:gamma}, a technical but otherwise straightforward check reveals that $s$ satisfies \eqref{eq:siota}, \eqref{eq:sj} and \eqref{eq:sGamma} if and only if $\gamma$ satisfies
	\begin{gather}
	Ti_X = i_{TX} \circ \gamma_X^\tensUnit, \label{eq:gammai} \\
	j_M \circ \mu_\tensUnit = [M,\mu_M] \circ \gamma^M_M \circ Tj_M, \label{eq:gammaj}
	\end{gather}
	and
	\begin{equation}\label{eq:gammaGamma}
	\Gamma_{TX,TY}^P \circ \gamma_Y^{TX} = [[P,TX],\gamma^P_Y] \circ \gamma^{[P,TX]}_{[P,Y]} \circ T\Gamma^P_{TX,Y},
	\end{equation}
	respectively, for all $(M,\mu_M)$, $(P,\mu_P)$ in $\cat^T$ and all $X,Y$ in $\cat$.
	\end{remark}

\subsection{The skew-closed skew-monoidal setting}\label{ssec:closedmonsetting}
  Let now $\cat$ be, in addition, left skew-closed skew-monoidal.
  To help spell out the intertwining condition in the examples, we employ the skew-closed
  skew-monoidal structure and 
  we denote by $L_X$ the left adjoint of $[X, -]$.
  The adjunction establishes a bijection between natural transformations $T[TX, -] \To
  [X, T-]$ and natural transformations $L_X T - \To T L_{T X} -$.
  Explicitly, the latter sends a natural transformation $s_{X, Y} \colon T[TX, Y]
  \to [X, TY]$ to $t_{X,Y} \colon L_Y T X \to T L_{TY} X$ given by
  \begin{equation}
    \label{eq:t_from_s}
    L_Y T X
    \xrightarrow{L_Y T \coev^{TY}_X} L_Y T [TY, L_{TY} X]
    \xrightarrow{L_Y s_{Y, L_{TY} X}} L_Y [Y, TL_{TY} X]
    \xrightarrow{\ev^Y_{T(L_{TY} X)}} T L_{TY} X
    \ .
  \end{equation}
  This yields a natural transformation $L_- (T ({\sim})) \To T (L_{T -} ({\sim}))$ between
  bifunctors, which we shall call the \emph{mate} of $s$.
	Notice that $t_{X,Y} \in \cat\big(L_YTX,T(L_{TY}X)\big)$ is the unique morphism such that
	\begin{equation}\label{eq:ts}
	[Y,t_{X,Y}] \circ \coev_{TX}^{Y} = s_{Y,L_{TY}X} \circ T\big(\coev_{X}^{TY}\big).
	\end{equation}
	
	\begin{remark}
	For every object $X$ in $\cat$,
	\[
	t_{X,-}\colon 
	\begin{gathered}
	\xymatrix @=30pt{
	\cat \ar[r]^-{T} \ar[d]_-{L_{TX}} & \cat \ar[d]^-{L_X} \\
	\cat \ar[r]_-{T} & \cat \ar@{}[ul]|-{\rotatebox[origin=c]{-45}{$\Downarrow$}}
	}
	\end{gathered}
	\quad \text{ is indeed the mate of } \quad
	s_{X,-}\colon 
	\begin{gathered}
	\xymatrix @=30pt{
	\cat \ar[r]^-{T} & \cat \\
	\cat \ar[r]_-{T} \ar[u]^-{[TX,-]} & \cat \ar[u]_-{[X,-]} \ar@{}[ul]|-{\rotatebox[origin=c]{45}{$\Downarrow$}}
	}
	\end{gathered}
	\] 
	in the categorical sense; one would just need to adapt the definition of a mate as in \cite[\S6.1, page 186]{Leinster} to fit the setting of \cite[Chapter IV, \S7, page 98]{Maclane}. Thus, our ``parametrized'' version of the terminology is justified.
	\end{remark}
	
	Now, recall from \cref{rem:closed_monoidal_category}\ref{item:closed3} the definitions of $\alpha,\lambda,\rho$ in terms of $i,j,\Gamma$.
	The properties of $s$ translate to the mate in the following way.

  \begin{proposition}
    \label{prop:recasting_s_to_t}
    Let $s_{X, Y} \colon T[TX, Y] \to [X, TY]$ be natural, and let $t$ be its mate.
    Then 
    \begin{enumerate}[leftmargin=0.8cm,label=(\arabic*),ref=(\arabic*)]
      \item
        $s$ is unital, i.e.\ satisfies \eqref{eq:lifting_closed_unitality}, if and only
        if $t$ satisfies
        \begin{align}
          \label{eq:t_unitality}
          t_{X,Y} \circ L_Y u_X
          = u_{L_{TY} X} \circ L_{u_Y} X 
        \end{align}
				for all $X,Y$ in $\cat$;

      \item
        $s$ is multiplicative, i.e.\ satisfies
        \eqref{eq:lifting_closed_multiplicativity}, if and only if $t$ satisfies
        \begin{align}
          \label{eq:t_multiplicativity}
          t_{X,Y} \circ L_Y m_X 
          = T(L_{m_Y} X) \circ m_{L_{T^2 Y} X} \circ T t_{X,TY} \circ t_{TX,Y}
        \end{align}
				for all $X,Y$ in $\cat$;
				
				\item $s$ satisfies \eqref{eq:siota} if and only if $t$ satisfies
					\begin{equation}\label{eq:trhoL}
					TL_{\mu_\tensUnit}X \circ t_{X,\tensUnit} \circ \rho_{TX} = T\rho_X
					\end{equation}
				  for all $X$ in $\cat$;
					
				\item $s$ satisfies \eqref{eq:sj} if and only if $t$ satisfies
	\begin{equation}\label{eq:tlambdaL}
	\lambda_M \circ L_M\mu_\tensUnit = \mu_M \circ m_M \circ T(\lambda_{TM}) \circ t_{\tensUnit,M}
	\end{equation}
	for all $(M,\mu_M)$ in $\cat^T$;
	
	\item\label{item:Prop3.6(5)} $s$ satisfies \eqref{eq:sGamma} if and only if $t$ satisfies
	\begin{equation}\label{eq:talphaL}
	\begin{aligned}
	& t_{X,L_MY} \circ \alpha_{TX,Y,M} = T\Big(L_{T\big(L_{\mu_M^{(2)}}Y \big)}X \Big) \circ T\left(L_{t_{Y,TM}}X\right) \circ T\left(\alpha_{X,TY,TM} \right) \circ t_{L_{TY}X,M} \circ L_Mt_{X,Y}
	\end{aligned}
	\end{equation}
	for all $X,Y$ in $\cat$ and all $(M,\mu_M)$ in $\cat^T$, where $\mu_M^{(2)}$ denotes $\mu_M \circ m_M = \mu_M \circ T\mu_M$.
    \end{enumerate}
  \end{proposition}

  \begin{proof}
		The proof is a technical but otherwise straightforward exercise.
    The dinaturality (see~\Cref{prop:ev_coev_dinatural}) of $\coev^X$ and $\ev^X$ is to
    be used.
		
		Nevertheless, since Claim \ref{item:Prop3.6(5)} is not elementary, let us sketch the underlying argument for the benefit of the reader. First of all, in view of the bijection 
	\[\cat\big(L_PL_{[P,X]}(T[TX,Y]), TY\big) \cong \cat\big(T[TX,Y], [[P,X],[P,TY]]\big)\]
	there exists a unique morphism $F$ on the left such that
	\[[[P,X],[P,F]] \circ [[P,X],\coev^P_{L_{[P,X]}T[TX,Y]}] \circ \coev^{[P,X]}_{T[TX,Y]} = \Gamma_{X,TY}^Z \circ s_{X,Y}\]
	and we claim that
	\[F = T(\ev^{TX}_Y) \circ T\big(L_{T(\ev^P_X)}[TX,Y] \big) \circ t_{[TX,Y],L_P[P,X]} \circ \alpha_{T[TX,Y],[P,X],P}.\]
	We check this by direct computation:
	{\small
	\begin{align*}
	&\begin{aligned}
	& [[P,X],[P,T(\ev^{TX}_Y)]] \circ [[P,X],[P,T(L_{T(\ev^P_X)}[TX,Y] )]] \circ [[P,X],[P,t_{[TX,Y],L_P[P,X]}]] ~ \circ \\
	& \circ ~ [[P,X],[P,\alpha_{T[TX,Y],[P,X],P}]] \circ [[P,X],\coev^P_{L_{[P,X]}T[TX,Y]}] \circ \coev^{[P,X]}_{T[TX,Y]} =
	\end{aligned} \\
	& \begin{aligned}
	\stackrel{\eqref{eq:alphaunique}}{\mathmakebox[\widthof{$\stackrel{\eqref{eq:ts}}{=}$}]{=}} ~
	& [[P,X],[P,T(\ev^{TX}_Y)]] \circ [[P,X],[P,T(L_{T(\ev^P_X)}[TX,Y] )]] \circ [[P,X],[P,t_{[TX,Y],L_P[P,X]}]] ~ \circ \\
	& \circ ~ [\coev_{[P,X]}^P,[P,L_{L_P([P,X])}T[TX,Y]]] \circ \Gamma_{L_P([P,X]),L_{L_P([P,X])}T[TX,Y]}^P \circ \coev_{T[TX,Y]}^{L_P([P,X])}
	\end{aligned} \\
	& \begin{aligned}
	\stackrel{(*)}{\mathmakebox[\widthof{$\stackrel{\eqref{eq:ts}}{=}$}]{=}} ~
	& [\coev_{[P,X]}^P,[P,TY]] \circ \Gamma_{L_P([P,X]),TY}^P \circ [L_P([P,X]),T(\ev^{TX}_Y)] ~ \circ \\ 
	& \circ ~ [L_P([P,X]),T(L_{T(\ev^P_X)}[TX,Y] )] \circ [L_P[P,X],t_{[TX,Y],L_P[P,X]}] \circ \coev_{T[TX,Y]}^{L_P([P,X])}
	\end{aligned} \\
	& \begin{aligned}
	\stackrel{\eqref{eq:ts}}{\mathmakebox[\widthof{$\stackrel{\eqref{eq:ts}}{=}$}]{=}} ~
	& [\coev_{[P,X]}^P,[P,TY]] \circ \Gamma_{L_P([P,X]),TY}^P \circ [L_P([P,X]),T(\ev^{TX}_Y)] ~ \circ \\ 
	& \circ ~ [L_P([P,X]),T(L_{T(\ev^P_X)}[TX,Y] )] \circ s_{L_P([P,X]),L_{T(L_P([P,X]))}[TX,Y]} \circ T\left(\coev_{[TX,Y]}^{T(L_P([P,X]))}\right)
	\end{aligned} \\
	& \begin{aligned}
	\stackrel{(*)}{\mathmakebox[\widthof{$\stackrel{\eqref{eq:ts}}{=}$}]{=}} ~
	& [\coev_{[P,X]}^P,[P,TY]] \circ \Gamma_{L_P([P,X]),TY}^P \circ s_{L_P([P,X]),Y} \circ T\big[T\big(L_P([P,X])\big),\ev^{TX}_Y\big] ~ \circ \\ 
	& \circ ~ T\big[T\big(L_P([P,X])\big),L_{T(\ev^P_X)}[TX,Y]\big] \circ T\left(\coev_{[TX,Y]}^{T(L_P([P,X]))}\right)
	\end{aligned} \\
	& \begin{aligned}
	\stackrel{(\star)}{\mathmakebox[\widthof{$\stackrel{\eqref{eq:ts}}{=}$}]{=}} ~
	& [\coev_{[P,X]}^P,[P,TY]] \circ \Gamma_{L_P([P,X]),TY}^P \circ s_{L_P([P,X]),Y} \circ T\big[T\big(L_P([P,X])\big),\ev^{T(L_P[P,X])}_Y\big] ~ \circ \\ 
	& \circ ~ T\big[T\big(L_P([P,X])\big),L_{T(L_P[P,X])}[T(\ev^P_X),Y]\big] \circ T\left(\coev_{[TX,Y]}^{T(L_P([P,X]))}\right)
	\end{aligned} \\
	& \begin{aligned}
	\stackrel{(*)}{\mathmakebox[\widthof{$\stackrel{\eqref{eq:ts}}{=}$}]{=}} ~
	& [\coev_{[P,X]}^P,[P,TY]] \circ \Gamma_{L_P([P,X]),TY}^P \circ s_{L_P([P,X]),Y} \circ T\big[T\big(L_P([P,X])\big),\ev^{T(L_P[P,X])}_Y\big] ~ \circ \\ 
	& \circ ~ T\left(\coev_{[T(L_P[P,X]),Y]}^{T(L_P([P,X]))}\right) \circ T\big[T(\ev^P_X),Y\big]
	\end{aligned} \\
	& \begin{aligned}
	\stackrel{\phantom{\eqref{eq:ts}}}{=} [\coev_{[P,X]}^P,[P,TY]] \circ \Gamma_{L_P([P,X]),TY}^P \circ s_{L_P([P,X]),Y} \circ T\big[T(\ev^P_X),Y\big] 
	\end{aligned} \\
	& \begin{aligned}
	\stackrel{(*)}{\mathmakebox[\widthof{$\stackrel{\eqref{eq:ts}}{=}$}]{=}} [\coev_{[P,X]}^P,[P,TY]] \circ T\big[T(\ev^P_X),Y\big] \circ \Gamma_{X,TY}^P \circ s_{X,Y} = \Gamma_{X,TY}^P \circ s_{X,Y}
	\end{aligned}
	\end{align*}
	}%
	where $(*)$ follow by naturality of the morphisms involved and $(\star)$ by dinaturality of $\ev$. Similarly, one can check that the unique morphism $G$ such that
	\begin{multline*}
	[[P,X],[P,G]] \circ [[P,X],\coev^P_{L_{[P,X]}T[TX,Y]}] \circ \coev^{[P,X]}_{T[TX,Y]} \\
	= [[\mu_P, X], s_{P, Y}] \circ s_{[TP, X], [TP, Y]} 
          \circ T[s_{P, X}, [\mu_P, Y]] \circ T \Gamma^P_{TX, Y}
	\end{multline*}
	is
	\[
	\begin{aligned}
	G = T(\ev^{TX}_Y) & \circ T\big(L_{T(\ev^P_X)}[TX,Y] \big) \circ T\Big(L_{T\big(L_{\mu_P^{(2)}}{[P,X]} \big)}{[TX,Y]} \Big) \circ T\left(L_{t_{{[P,X]},TP}}{[TX,Y]}\right) \circ \\
	& \circ T\left(\alpha_{{[TX,Y]},T{[P,X]},TP} \right) \circ t_{L_{T{[P,X]}}{[TX,Y]},P} \circ L_P\left(t_{[TX,Y],{[P,X]}}\right).
	\end{aligned}
	\]
	Therefore,
	\begin{equation}\label{eq:techtech}
	\begin{aligned}
	& T(\ev^{TX}_Y) \circ T\big(L_{T(\ev^P_X)}[TX,Y] \big) \circ t_{[TX,Y],L_P[P,X]} \circ \alpha_{T[TX,Y],[P,X],P} \\
	& \begin{aligned} = T(\ev^{TX}_Y) & \circ T\big(L_{T(\ev^P_X)}[TX,Y] \big) \circ T\Big(L_{T\big(L_{\mu_P^{(2)}}{[P,X]} \big)}{[TX,Y]} \Big) \circ T\left(L_{t_{{[P,X]},TP}}{[TX,Y]}\right) \circ \\
  & \circ T\left(\alpha_{{[TX,Y]},T{[P,X]},TP} \right) \circ t_{L_{T{[P,X]}}{[TX,Y]},P} \circ L_P\left(t_{[TX,Y],{[P,X]}}\right) \end{aligned}
	\end{aligned}
	\end{equation}
	for every $X,Y$ in $\cat$, $(P,\mu_P)$ in $\cat^T$. Now, 
	{\scriptsize
	\begin{align*}
	& t_{X,L_MY} \circ \alpha_{TX,Y,M} \\
	& \mathmakebox[\widthof{$\stackrel{\eqref{eq:techtech}}{=}$}]{=} T\left(\ev_{L_{TL_MY}X}^{TL_MY}\right) \circ t_{[TL_MY,L_{L_MY}X],L_MY} \circ \alpha_{T[TL_MY,L_{TL_MY}X],Y,M} \circ L_ML_YT\coev_{X}^{TL_MY}\\
	& \begin{aligned} \mathmakebox[\widthof{$\stackrel{\eqref{eq:techtech}}{=}$}]{=} ~ & T\left(\ev^{TL_MY}_{L_{TL_MY}X}\right) \circ T\big(L_{T(\ev^M_{L_MY})}[TL_MY,{L_{TL_MY}X}] \big) \circ t_{[TL_MY,{L_{TL_MY}X}],L_M[M,L_MY]} ~ \circ \\
	& \circ ~\alpha_{T[TL_MY,{L_{TL_MY}X}],[M,L_MY],M} \circ L_ML_{\coev^M_Y}T[T(L_MY),{L_{TL_MY}X}] \circ L_ML_YT\coev_{X}^{TL_MY}
	\end{aligned} \\
	& \begin{aligned} \stackrel{\eqref{eq:techtech}}{\mathmakebox[\widthof{$\stackrel{\eqref{eq:techtech}}{=}$}]{=}} ~ & T\left(\ev^{TL_MY}_{L_{TL_MY}X}\right) \circ T\big(L_{T(\ev^M_{L_MY})}[TL_MY,{L_{TL_MY}X}] \big) \circ T\Big(L_{T\big(L_{\mu^{(2)}_M}[M,L_MY]\big)}[TL_MY,L_{TL_MY}X]\Big) ~ \circ \\
	& \circ ~ T\left(L_{t_{[M,L_MY],TM}}[TL_MY,L_{TL_MY}X]\right) \circ T\left(\alpha_{[TL_MY,{L_{TL_MY}X}],T[M,L_MY],TM}\right) \circ t_{L_{T[M,L_MY]}[TL_MY,L_{TL_MY}X],M} ~ \circ \\ 
	& \circ ~ L_M\left(t_{[TL_MY,L_{TL_MY}X],[M,L_MY]}\right) \circ L_ML_{\coev^M_Y}T[T(L_MY),{L_{TL_MY}X}] \circ L_ML_YT\coev_{X}^{TL_MY}
	\end{aligned} \\
	& \begin{aligned} \stackrel{(*)}{\mathmakebox[\widthof{$\stackrel{\eqref{eq:techtech}}{=}$}]{=}} ~ & T\left(\ev^{TL_MY}_{L_{TL_MY}X}\right) \circ  T\Big(L_{T\big(L_{\mu^{(2)}_M}Y\big)}[TL_MY,L_{TL_MY}X]\Big) \circ T\left(L_{t_{Y,TM}}[TL_MY,L_{TL_MY}X]\right) ~ \circ \\
	& \circ ~ T\left(\alpha_{[TL_MY,{L_{TL_MY}X}],TY,TM}\right) \circ t_{L_{TY}[TL_MY,L_{TL_MY}X],M} \circ L_M\left(t_{[TL_MY,L_{TL_MY}X],Y}\right) \circ L_ML_YT\coev_{X}^{TL_MY}
	\end{aligned} \\
	& \begin{aligned} \stackrel{(*)}{\mathmakebox[\widthof{$\stackrel{\eqref{eq:techtech}}{=}$}]{=}} T\Big(L_{T\big(L_{\mu^{(2)}_M}Y\big)}X\Big) \circ T\left(L_{t_{Y,TM}}X\right) \circ T\left(\alpha_{X,TY,TM}\right) \circ t_{L_{TY}X,M} \circ L_M\left(t_{X,Y}\right),
	\end{aligned}
	\end{align*}
	}%
	where $(*)$ follow by naturality again. We leave to the interested reader to check that also the other implication holds. \qedhere
  \end{proof}
	
	\begin{remark}\label{rem:rephrasingt}
	Let us take advantage of the convention $X \otimes Y \coloneqq L_Y(X)$ as we did in \cref{ssec:skewclosed}. Then $t_{X,Y} \colon TX \otimes Y \to T(X \otimes TY)$. With this notation, \eqref{eq:t_unitality} becomes
	\begin{equation}\label{eq:t_unitality_tens}
	t_{X,Y} \circ (u_X \otimes Y) = u_{X \otimes TY} \circ (X \otimes u_Y)
	\end{equation}
	and \eqref{eq:t_multiplicativity} becomes
	\begin{equation}\label{eq:t_multiplicativity_tens}
	t_{X,Y} \circ (m_X \otimes Y) = T(X \otimes m_Y) \circ m_{X \otimes T^2Y} \circ Tt_{X,TY} \circ t_{TX,Y}.
	\end{equation}
	Analogously, \eqref{eq:trhoL}, \eqref{eq:tlambdaL} and \eqref{eq:talphaL} become
	\begin{gather}
	T(X \otimes \mu_\tensUnit) \circ t_{X,\tensUnit} \circ \rho_{TX} = T\rho_X, \label{eq:trho} \\
	\lambda_M \circ (\mu_\tensUnit \otimes M) = \mu_M \circ m_M \circ T(\lambda_{TM}) \circ t_{\tensUnit,M}, \label{eq:tlambda} \\
	\begin{aligned}
	t_{X,Y \otimes M} \circ \alpha_{TX,Y,M} ~=~ & T\big(X \otimes T(Y \otimes \mu_M^{(2)})\big) \circ T\left(X \otimes t_{Y,TM}\right) \circ T\left(\alpha_{X,TY,TM} \right) \\ & ~\circ~ t_{X \otimes TY,M} \circ \left(t_{X,Y} \otimes M\right)
	\end{aligned} \label{eq:talpha}
	\end{gather}
	respectively, for all $X,Y$ in $\cat$ and $(M,\mu_M)$ in $\cat^T$ and where $\mu_M^{(2)}$ denotes $\mu_M \circ m_M = \mu_M \circ T\mu_M$.
	\end{remark}

  Let us now write the action of $T$ on $[-, -]$ in terms of $t$, for later use.
  To this end, recall that given $T$-algebras $(M, \mu_M)$ and $(N, \mu_N)$, the object $[M, N]$
  becomes a $T$-algebra with action $\mu_M \star \mu_N = [M, \mu_N] \circ s_{M, N} \circ T [\mu_M, N]$.
  In terms of $t$, one finds that the $T$-algebra structure is the composition
  \begin{equation}
	\begin{gathered}
    \xymatrix @C=50pt{
		T[M, N] \ar@{.>}[d]_-{\mu_{[M,N]}} \ar[r]^-{\coev^M_{T[M, N]}} & [M, L_M T[M, N]] \ar[r]^-{[M, t_{[M,N],M}]} & [M, T L_{TM} [M, N]] \ar[d]^-{[M, T L_{\mu_M} [M, N]]} \\
    [M, N] & [M, TN] \ar[l]^-{[M, \mu_N]} & [M, T L_M [M, N]] \ar[l]^-{[M, T \ev^M_N]} \ .
		}
		\end{gathered}
    \label{eq:T_action_using_t}
  \end{equation}

	\begin{remark}
	Continuing \cref{rem:gamma2}, the mate $\tau$ of $\gamma$ (and in this case $\tau^M$ is exactly the mate of $\gamma^M$ in the sense of \cite[\S6.1, page 186]{Leinster}, for every $M$ in $\cat^T$) is the natural transformation $\tau^M_X \colon L_MT(X) \to TL_M(X)$ given by
	\[L_MT(X) \xrightarrow{L_MT\coev^M_X} L_MT[M,L_M(X)] \xrightarrow{L_M\gamma^M_{L_M(X)}} L_M[M,TL_M(X)] \xrightarrow{\ev^M_{TL_M(X)}} TL_M(X).\]
	The natural transformation $\gamma$ satisfies \eqref{eq:gamma1} and \eqref{eq:gamma2} if and only if $\tau$ satisfies
	\[\tau_{X}^M \circ (u_X \otimes M) = u_{X \otimes M} \qquad \text{and} \qquad \tau^M_X \circ (m_X \otimes M) = m_{X \otimes M} \circ T\tau_X^M \circ \tau^M_{TX},\]
	respectively.
	In terms of $\tau$, the $T$-algebra structure on $[M,N]$ looks like
	\[
	\xymatrix @C=40pt @R=18pt {
	T[M,N] \ar@{.>}[dd]_-{\mu_{[M,N]}} \ar[r]^-{\coev_{T[M,N]}^M} & [M,T[M,N] \otimes M] \ar[d]^-{[M,\tau^M_{[M,N]}]} \\
	& [M,T([M,N]\otimes M)] \ar[d]^-{[M,T(\ev_N^M)]} \\
	[M,N] & [M,TN] \ar[l]^-{[M,\mu_N]} 
	}
	\]
	(by replacing $\tau$ with its definition, we find the action via $\gamma$). 
	It follows that $\gamma$ satisfies \eqref{eq:gammai} and \eqref{eq:gammaj} if and only if $\tau$ satisfies
	\begin{gather*}
	T \rho_X = \tau_X^{\tensUnit} \circ \rho_{TX}, \\
	\lambda_M \circ (\mu_\tensUnit \otimes M) = \mu_M \circ T\lambda_M \circ \tau_M^\tensUnit,
	\end{gather*}
	respectively. Furthermore, $\gamma$ satisfies \eqref{eq:gammaGamma} if and only if $\tau$ satisfies
	\[
	\tau_X^{T(Y \otimes M)} \circ \big(TX \otimes \tau_Y^M\big) \circ \alpha_{TX,TY,M} = T\big(X \otimes \tau^M_Y\big) \circ T \big(\alpha_{X,TY,M}\big) \circ \tau_{X \otimes TY}^M \circ \big(\tau_X^{TY} \otimes M\big)
	\]
	for all $X,Y$ in $\cat$ and $(M,\mu_M)$ in $\cat^T$.
	\end{remark}

  \begin{example}
    \label{ex:Hopf_monad_2}
    Let $T$ be a left Hopf monad on the closed monoidal category $\cat$, with left fusion
    operator $H^l_{A, B} \colon T(A \tensor T B) \to TA \tensor TB$.
    Choose $L_X = - \tensor X$.
    Then the natural transformation defined as the composition
    \begin{align*}
      t_{A,B}\colon 
        L_B TA = TA \tensor B \xrightarrow{TA \tensor u_B}
        TA \tensor TB \xrightarrow{(H^l)\inv_{A, B}}
        T(A \tensor TB) = T L_{T B} A
    \end{align*}
    satisfies unitality \eqref{eq:t_unitality} and multiplicativity \eqref{eq:t_multiplicativity} as in \Cref{prop:recasting_s_to_t}.
    This can be seen from the formulas in \cite[Propositions 3.8 and 3.9]{BLV-hopf_monads}.
    We can re-express $(H^l)\inv_{A, B}$ in terms of $t$ as the composition
    \begin{align*}
      TA \tensor TB \xrightarrow{t_{A,TB}}
      T(A \tensor T^2 B) \xrightarrow{T(A \tensor m_B)}
      T(A \tensor T B)
      \ .
    \end{align*}
    This observation implicitly appeared in \cite[Proposition 3.9]{BLV-hopf_monads}.
  \end{example}

\subsection{Set-theoretic solutions}
\label{ssec:gabi-algebras_in_set}

  We continue \Cref{ex:set_is_closed}, the cartesian closed category $\Set$.
	As we are used to, take $L_X \coloneqq - \times X$.
  Let $M$ be an algebra in $\Set$, i.e.\ a monoid, and consider the monad $T = M \times
  -$.
  We write multiplication in $M$ by juxtaposition, the unit is $1$. The action of $M$ on an $M$-set $U$ will be denoted by $M \times U \to U, (m,u) \mapsto m.u$.

  Since an element in a set $S$ is the same thing as function $* \to S$ from the point, a
  natural transformation $t_{X, Y} \colon (M \times X) \times Y \to M \times (X \times (M \times Y))$ as in \eqref{eq:t_from_s} is completely determined by its value $t_{*, *}$.
  Thus, such a natural transformation is completely determined by a function $\delta
  \colon M \to M \times M$, $m \mapsto (m_+, m_-)$, via $t_{X, Y}(m, x, y) = (m_+, x, m_-, y)$.

  Unitality of $t$ means that
  \begin{align*}
    (1_+, x, 1_-, y)
    &= \big(t_{X, Y} \circ (u_X \times Y)\big)(x, y)
    ~\oversetEq[{\eqref{eq:t_unitality_tens}}]~
    \big(u_{X \times M \times Y} \circ (X \times u_Y)\big)(x, y)
    = (1, x, 1, y)
    \ ,
  \end{align*}
  while multiplicativity means
  \begin{align*}
    & ((mn)_+, x, (mn)_-, y) = \big(t_{X, Y} \circ (m_X \times Y)\big)(m, n, x, y)
    \\ & \stackrel{\eqref{eq:t_multiplicativity_tens}}{=} 
    ((M \times X \times m_Y) \circ m_{X \times M \times M \times Y} \circ (M \times t_{X, M \times Y}) \circ t_{M \times X, Y})(m, n, x, y)
    \\ & \stackrel{\phantom{\eqref{eq:t_multiplicativity_tens}}}{=}
    (m_+ n_+, x, n_- m_-, y)
  \end{align*}
	for all sets $X,Y$, all $x \in X$, $y \in Y$ and $m,n \in M$.
  Thus we get a lift of the inner hom functor of $\Set$ if and only if $\delta \colon M \to M \times
  M\op$ is a monoid map.
  Here $M\op$ is the opposite monoid, i.e.\ the set $M$ with multiplication $m \cdot n =
  n m$ and unit element $1$.

  Given two $M$-sets $U$ and $V$, from \eqref{eq:T_action_using_t} one computes the action of $M$ on $[U, V] = \Set(U, V)$
  to be
	\[(m. f)(u) = m_+ . f (m_- . u)\]
	for all $m \in M$, $u \in U$, $f \in \Set(U,V)$.
  The unique function $\counit \colon M \to *$ to the terminal object serves as an
  augmentation: it is automatically an algebra map.
	
  To lift the full \emph{skew-closed} structure of $\Set$ to $\prescript{}{M}{\Set} = \Set^{M
  \times -}$, we need the transformations $i, j, \Gamma$ specified in
  \Cref{ex:set_is_closed} to be $M$-module morphisms.
	To this aim, $t$ satisfies \eqref{eq:trho} (up to identifying $X \times *$ with $X$) if and only if
	\begin{align*}
	(m_+, x) & = \big((M \times X \times \counit) \circ t_{X,*} \circ \rho_{M \times X}\big)(m,x) \\
	& = (M \times \rho_X)(m,x) = (m,x)
	\end{align*}	
	for every set $X$, all $x \in X$ and all $m \in M$,
	which is equivalent to $\delta(m) = (m, m_-)$. Furthermore, $t$ satisfies \eqref{eq:tlambda} (up to identifying $* \times X$ with $X$) if and only if
	\[
		u = \big(\lambda_U \circ (\counit \times U)\big)(m,u) = \big(\mu_U \circ m_U \circ (M \times \lambda_{M \times U}) \circ t_{*,U}\big)(m,u) = m_+m_-.u
	\]
	for each $M$-set $U$ and for every $u \in U$.
  This yields the necessary and sufficient condition $m_+ m_- = 1$ for all $m \in M$.

  Before discussing $\Gamma$, we note that this already forces $M$ to be a group,
  possibly subject to extra conditions imposed by $\Gamma$.
  Indeed, $m m_- = m_+ m_- = 1$ says that every element $m \in M$ has a right inverse.
  In particular, for a given $m$, we have that $m_-$ has $m_{--}$ as right inverse.
  Since $M$ is associative, the simple computation
  \begin{align*}
    m = m (m_- m_{--}) = (m m_-) m_{--} = m_{--}
  \end{align*}
  shows that $m$ is a two-sided inverse of $m_-$, and so every element of $M$ is
  invertible.

	For $\Gamma$, we have that $t$ satisfies \eqref{eq:talpha} if and only if
	\begin{align*}
	(m_+,x,m_-,y,u) & \stackrel{\phantom{\eqref{eq:talpha}}}{=} \big(t_{X,Y \times U} \circ \alpha_{M \times X, Y , U}\big)(m,x,y,u) \\
	& \begin{aligned}
	~ \stackrel{\eqref{eq:talpha}}{=} ~ & \Big(\big(M \times X \times M \times Y \times \mu_U^{(2)}\big) \circ (M \times X \times t_{Y,M \times U}) \circ (M \times \alpha_{X,TY,TU}) \\ 
		& ~\circ~ t_{X \times M \times Y,U} \circ \left(t_{X,Y} \times U\right) \Big)(m,x,y,u)
	\end{aligned} \\
	& \stackrel{\phantom{\eqref{eq:talpha}}}{=} (m_{++},x,m_{-+},y,m_{--}m_{+-}.u)
	\end{align*}
	for all sets $X,Y$, all $M$-sets $(U,\mu_U)$, and for every $x \in X$, $y \in Y$, $u \in U$ and $m \in M$. This leads to $(m_{++}, m_{-+}, m_{--} m_{+-}) = (m_+, m_-, 1)$ for all $m \in M$.	

  Finally, this last condition is also satisfied for groups, i.e.\ if $(m_+, m_-) = (m,
  m\inv)$.
  Indeed,
  \begin{align*}
    (m_{++}, m_{-+}, m_{--} m_{+-})
    = (m, m\inv, (m\inv)\inv m\inv)
    = (m, m\inv, 1)
    = (m_+, m_-, 1)
    \ .
  \end{align*}

  Similar considerations can be done for the monad $- \times M$ and we summarize this
  example in the following \nameCref{prop:gabi-monoids_in_set_are_groups}.

  \begin{theorem}
    \label{prop:gabi-monoids_in_set_are_groups}
    Let $M \in \Set$ be a monoid.
    The following are equivalent:
    \begin{enumerate}[leftmargin=0.8cm]
      \item
        $M$ is a group.

      \item
        The category $\prescript{}{M}{\Set}$ $(\Set_M)$ of left (right) $M$-sets is
        left skew-closed such that the forgetful functor to $\Set$ is strictly closed.
    \end{enumerate}
  \end{theorem}
  
\subsection{The linear case}
\label{ssec:gabi-algebras_in_kmod}
  We now look at solutions in $\hmodM[k]$ for a commutative ring $k$, following up on
  \Cref{ex:k-mod_is_closed}.
  Here $L_M = - \tensor M$.

  For a $k$-algebra $A$, the functor $T = A \tensor -$ is a monad on
  $\hmodM[k]$ in a natural way.
  We need to exhibit a natural $k$-module map $t_{M, N} \colon (A \otimes M)
  \tensor N \to A \otimes (M \tensor (A \tensor N))$ with certain properties.
  Note that such a map is determined by $t_{k, k}$:
  we have
  \begin{align*}
    t_{M, N}(a \tensor m \tensor n)
    &= \big(t_{M, N} \circ (A \tensor c_M(m) \tensor c_N(n))\big) (a \tensor 1 \tensor 1)
    \\
    &= \big((A \tensor c_M(m) \otimes A \tensor c_N(n)) \circ t_{k, k}\big) 
    (a \tensor 1 \tensor 1)
  \end{align*}
  by naturality, where $c_M(m) \colon k \to M$ is the unique $k$-linear map such that $c_M(m)(1_k) = m$ for all $m \in M$.
  We define, in a Sweedler-like notation,
	\[\delta \colon A \to A \tensor A, \qquad a \mapsto a_+ \otimes a_-,\] 
	(summation understood) to be the obvious map such that $t_{k,k}(a) = a_+ \tensor 1 \tensor a_- \tensor 1$.

  From here on, one carries out, \emph{mutatis mutandis}, the computations from
  \cref{ssec:gabi-algebras_in_set}. That is, one declines the formulas from \cref{rem:rephrasingt} in this specific situation.
  We leave the details to the reader, and instead just state:

  \begin{theorem}[\cite{Boehm-private}]
    \label{prop:gabi_algebra_defining_properties}
    Let $A$ be an algebra over a commutative ring $k$.
    Then the closed structure of $\hmodM[k]$ lifts to a left skew-closed structure on $\hmodM[A]$ if and only if:
    \begin{enumerate}[leftmargin=0.8cm]
      \item
        $(A, \counit)$ is an augmented $k$-algebra,
      \item
        there exists an algebra map $\delta \colon A \to A \tensor A\op$, $\delta(a) =
        a_+ \tensor a_-$ (summation understood), 
    \end{enumerate}
		and they satisfy
		\begin{gather}
		a_+ \counit(a_-) = a \label{eq:gabi_algebra_counitality} \tag{GA1} \\
		a_+ a_- = \counit(a) 1_A \label{eq:gabi_algebra_antipode_axioms} \tag{GA2} \\
		a_{++} \tensor a_{-+} \tensor a_{--} a_{+-} = a_+ \tensor a_- \tensor 1 \label{eq:gabi_algebra_associativity} \tag{GA3}
		\end{gather}
    In this case, $A$ acts on $\hmodM[k](M, V)$ as $(a . f)(m) = a_+ f (a_- m)$.
  \end{theorem}

  \begin{remark}\label{rem:varia}
  \begin{enumerate}[label=(\arabic*),ref=(\arabic*),leftmargin=0.8cm]
  \item
  Given the lifted skew-closed structure on $\hmodM[A]$, we can obtain the gabi-algebra structure on $A$ explicitly in the following way. The fact that $k$ is the unit of the closed structure makes that $k$ must be endowed with a left $A$-module structure. The augmentation map $\counit$ is then given by $\counit(a)=a\cdot 1_k$. Now consider the left $A$-modules $.A\otimes A$ and $A\ot .A$, where the dot indicates how $A$ acts regularly on the first and second tensorand, respectively. Since the closed structure is lifted from $\hmodM[k]$, we have that $[A\otimes .A,.A\ot A]=\hmodM[k](A\otimes .A,.A\otimes A)$, which contains the identity map $\id \colon A\otimes A\to A\otimes A$. Then the comultiplication map $\delta$ is given by 
  $$\delta(a)=(a.\id) (1_A\ot 1_A).$$
  \item
    \label{rem:gabi_algebras_opposite_version}
    The above \namecref{prop:gabi_algebra_defining_properties} details the case of left
    modules, i.e.\ we looked at algebras over the monad $A \tensor -$.
    This choice was arbitrary: we could have chosen the monad $- \tensor A$, whose
    algebras are exactly the right $A$-modules.
    One then finds that lifting the skew-closed structure of $\hmodM[k]$ to $\hmodM[][A]$ is
    equivalent to $A$ being augmented via $\counit$, and equipped with an algebra map
    $\delta' \colon A \to A\op \tensor A$, $a \mapsto a_- \tensor a_+$, satisfying similar
    properties as above.
    In detail, we have
		\begin{gather}
		\counit(a_-) a_+ = a, \tag{GA1'} \label{eq:gabi_algebra_counitality_bis} \\
		a_- a_+ = \counit(a) 1_A, \tag{GA2'} \label{eq:gabi_algebra_antipode_axioms_bis} \\
		a_{+-} a_{--} \tensor a_{-+} \tensor a_{++} = 1 \tensor a_- \tensor a_+ \tag{GA3'} \label{eq:gabi_algebra_associativity_bis}
		\end{gather}
    for all $a \in A$.
    In this case, $A$ acts on $\hmodM[k](M, V)$ as $(f . a)(m) = f(m a_-) a_+$.
  \item
  	Of course, one can work with coalgebras and formally dualise \eqref{eq:gabi_algebra_counitality}, \eqref{eq:gabi_algebra_antipode_axioms} and \eqref{eq:gabi_algebra_associativity} from \Cref{prop:gabi_algebra_defining_properties} to obtain what we would call a \emph{left gabi-coalgebra}, or \eqref{eq:gabi_algebra_counitality_bis}, \eqref{eq:gabi_algebra_antipode_axioms_bis} and \eqref{eq:gabi_algebra_associativity_bis} from \ref{rem:gabi_algebras_opposite_version} above to obtain a \emph{right gabi-coalgebra}. 
	
	As for algebras, a gabi-coalgebra structure on a given coalgebra leads to additional structure on its category of comodules. 
	For example, let $C$ be a right gabi-coalgebra, that is a coalgebra $C$ endowed with two coalgebra maps, a unit $1_C\colon k \to C$ and a multiplication $\nabla \colon C \ot C^{\mathrm{cop}}\to C$, $c \otimes d \mapsto c.d$, satisfying
  	\begin{gather}
  	 c.1_C = c, \tag{GC1'} \\
  	 c_{(1)}.c_{(2)}=\varepsilon(c)1_C, \tag{GC2'} \\
  	 \epsilon(e)c.d = \big(c.e_{(2)}\big).\big(d.e_{(1)}\big). \tag{GC3'}
  	\end{gather}
		Consider two right comodules $N$ and $P$. Then the $k$-linear hom space between them can be endowed with the following structure map
	\[
	\xymatrix @R=0pt{
	\hmodM[k](N,P)  \ar[r] & \hmodM[k](N,P \ot C) \\
	\varphi \ar@{|->}[r] & \bigg(n \mapsto \sum_i \varphi\big({n}_0\big)_0 \ot \varphi\big({n}_0\big)_1.{n}_1 \bigg)
	}
	\]
	If $N$ is moreover finite-dimensional, then we have a natural $k$-linear isomorphism $$\hmodM[k](N,P \ot C)\cong \hmodM[k](N,P) \ot C$$ and combined with this isomorphism, the above map endows $\hmodM[k](N,P)$ with a right $C$-comodule structure. In particular, this will endow the category of all finite-dimensional right $C$-comodules with a skew closed structure in such a way that the forgetful functor to finite-dimensional vector spaces strictly preserves the closed structure. 
	
	Using classical Tannaka duality, the above reasoning can be reversed. First, recall that for any finite-dimensional right comodule $N$ over an arbitrary coalgebra $C$, $N^*$ is a left $C$-comodule with respect to
	\[N^* \to C \ot N^*, \qquad f \mapsto \sum f\big({e_i}_0\big){e_i}_1 \ot e^i,\]
	where $\sum_i e_i \ot e^i$ is a dual basis of $N$, and so it is a right $C^{\mathrm{cop}}$-comodule with respect to
	\[N^* \to N^* \ot C^{\mathrm{cop}}, \qquad f \mapsto \sum e^i \ot f\big({e_i}_0\big){e_i}_1.\]
	This shows that the contravariant auto-equivalence on the category of finite-dimensional vector spaces given by taking the linear dual lifts to an contravariant equivalence between category of finite-dimensional right $C$-comodules and the category of finite-dimensional right $C^{\mathrm{cop}}$-comodules. Suppose now that the category ${\mathcal M}^C_{\mathrm{fd}}$ of finite-dimensional right $C$-comodules is skew closed in such a way that the forgetful functor to finite-dimensional vector spaces strictly preserves the closed structure. Then, by applying the above equivalence, the internal hom functor
	$$[-,-]:({\mathcal M}^C_{\mathrm{fd}})^{\mathrm{op}} \times {\mathcal M}_{\mathrm{fd}}^C \to {\mathcal M}^C_{\mathrm{fd}}$$
	can be rewritten as
	$$[-,-]:{\mathcal M}^{C^{\mathrm{cop}}}_{\mathrm{fd}} \times {\mathcal M}_{\mathrm{fd}}^C \to {\mathcal M}^C_{\mathrm{fd}}$$
	which translates by Tannaka duality to a coalgebra morphism
	$$\nabla:C\otimes C^{\mathrm{cop}}\to C$$
	which will endow $C$ with a right gabi-coalgebra structure.
	\qedhere
\end{enumerate}
  \end{remark}

\section{Gabi-algebras: basic properties and examples}\label{sec:gabialgebras_props_examples}

  For the rest of the paper, we will be concerned with the situation of
  \Cref{prop:gabi_algebra_defining_properties}.

\subsection{Definition and examples}\label{ssec:defandex}
  Let us now define the main character of the paper.
  \begin{definition}
    An algebra $A$ over $k$ satisfying the equivalent conditions of
    \Cref{prop:gabi_algebra_defining_properties} is called a \emph{left gabi-algebra}. 
    If $A$ satisfies the conditions from \Cref{rem:varia}\ref{rem:gabi_algebras_opposite_version}, then
    we call it a \emph{right gabi-algebra}. If $A$ is a gabi-algebra such that the lifted skew-closed structure on $\hmodM[A]$ is normal, then we call $A$ a {\em normal gabi-algebra}.
  \end{definition}

  Henceforth, we focus on the left-hand side case and we call left gabi-algebras simply
  \emph{gabi-algebras}.

  \begin{example}[Hopf algebras]
    \label{ex:hopf_algebras_are_gabi}
    Let $H$ be a Hopf algebra over $k$.
    Then $\delta(h) = h\sweedler{1} \tensor S(h\sweedler{2})$ turns $H$ into a
    gabi-algebra.
    Indeed, $H$ is augmented;
    $h\sweedler{1} \counit(S(h\sweedler{2})) = h$ holds since $S$ is an anti-coalgebra
    map and $\Delta$ is (right) counital;
    $h\sweedler{1} S(h\sweedler{2}) = \counit(h) 1$ holds since $S$ is a (right)
    convolution inverse of the identity;
    and finally
    \begin{align*}
      h\sweedler{1,1} \tensor & S(h\sweedler{2})\sweedler{1}
      \tensor S(S(h\sweedler{2})\sweedler{2}) S(h\sweedler{1,2})
      = h\sweedler{1,1} \tensor S(h\sweedler{2})\sweedler{1}
      \tensor S(h \sweedler{1, 2} S(h\sweedler{2})\sweedler{2}) 
      \\
      &= h\sweedler{1,1} \tensor S(h\sweedler{2,2})
      \tensor S(h\sweedler{1,2} S(h\sweedler{2,1}))
      = h\sweedler{1} \tensor S(h\sweedler{4})
      \tensor S(h\sweedler{2} S(h\sweedler{3}))
      \\
      &= 
      h\sweedler{1} \tensor S(h\sweedler{2}) \tensor 1
      \ .
    \end{align*}
    If the antipode of $H$ is invertible, then also the map $h \mapsto h\sweedler{2}
    \tensor S\inv(h\sweedler{1})$ turns $H$ into a
    gabi-algebra. 
    The choice of using the antipode or the inverse antipode corresponds to lifting
    either the (entire) right or left closed monoidal structure to the category of left
    modules.
    Evidently, $h \mapsto S(h\sweedler{1}) \tensor h\sweedler{2}$ yields a right
    gabi-algebra structure.
  \end{example}
	
	\begin{example}[one-sided Hopf algebras]\label{ex:onesidedHopf}
	In \cref{ex:hopf_algebras_are_gabi} we did not use the fact the antipode $S$ is a left
    convolution inverse of the identity.
    This suggests another class of examples, namely one-sided Hopf algebras.
		
    A \emph{right} (resp., \emph{left}) \emph{Hopf algebra} is a bialgebra $B$ in which the identity has a right
    (resp., left) convolution inverse.
    A right Hopf algebra whose right antipode is anti-multiplicative and 
		anti-comultiplicative carries the structure of a left
    gabi-algebra with respect to $\counit$ and $\delta(b) \coloneqq b\sweedler{1} \otimes S(b\sweedler{2})$ for all $b \in B$, i.e., it lifts the closed structure of $\hmodM[k]$ to its category of
    left modules (symmetrically for the left Hopf algebra case). It is easy to see that a one-sided convolution inverse
    of the identity is always unital and counital 
    (see, e.g., \cite[Remark 3.8]{Saracco-Frob}), so an anti-multiplicative and
    anti-comultiplicative one-sided antipode is 
		automatically an anti-bialgebra morphism.
		
    Examples of one-sided but not two-sided Hopf algebras exist in the literature: see \cite{Green-Nichols-Taft,Lauve-Taft,Nichols-Taft,Rodriguez-Taft}.
		Namely, \cite[Example 21]{Green-Nichols-Taft} exhibits an example of a genuine left Hopf algebra whose left antipode is an anti-bialgebra morphism, constructed as the free left Hopf algebra over a coalgebra. In \cite[\S3]{Nichols-Taft}, a similar construction is used to exhibit an example of a left Hopf algebra in which no left antipode can be an anti-bialgebra morphism. In \cite[\S3]{Rodriguez-Taft}, a new example of a left Hopf algebra is provided by modifying the construction of $SL_q(2)$; here as well the left antipode is neither and anti-algebra nor an anti-coalgebra morphism. The latter example is extended in \cite{Lauve-Taft} to provide a whole family of left Hopf algebras $\tilde{SL}_q(n)$ for all $n \geq 2$.
	\end{example}

Inspired by \cref{ex:hopf_algebras_are_gabi} and \cref{ex:onesidedHopf}, let us introduce the following definition.

  \begin{definition}\label{def:antipode}
    The \emph{antipode} of a gabi-algebra $A$ is the map $\sigma \colon A \to A$
    defined by $\sigma(a) = \counit(a_+) a_-$ for all $a \in A$.
  \end{definition}
	
	The antipode $\sigma$ is a composition of algebra morphisms, 
  \begin{align*}
    \sigma = A \xrightarrow{\delta} A \otimes A\op \xrightarrow{\counit \tensor \id} A\op
    \ ,
  \end{align*}
  and hence itself an algebra morphism.
  Furthermore, it is a morphism of augmented algebras since for all $a \in A$,
  \[(\counit \circ \sigma)(a) = \counit(a_+)\counit(a_-) = \counit(a).\]

  \begin{remark}
    One can now give trivial examples of an algebra $H$ such that the forgetful functor $\omega
    \colon \hmodM[H] \to \hmodM[k]$ is both strictly closed and strictly monoidal, but not
    closed monoidal.
    We simply need a bialgebra structure implementing a monoidal structure, and an
    unrelated Hopf algebra structure implementing the closed structure as in
    \Cref{ex:hopf_algebras_are_gabi}.
    For example, $H = k[x]$, with the bialgebra structure 
    \begin{align*}
      \Delta(x) = x \tensor x, \quad \counit(x) = 1,
    \end{align*}
    and the Hopf algebra structure
    \begin{equation*}
      \Delta(x) = 1 \tensor x + x \tensor 1, \quad \counit(x) = 0, \quad S(x) = -x.
    \end{equation*}
		A more significant class of examples is offered by one-sided Hopf algebras as in \cref{ex:onesidedHopf}. We will treat the latter case in more detail in \cref{ssec:rightHopf}.
  \end{remark}

	In view of what we mentioned at the beginning of \cref{ssec:EMcats}, for a gabi-algebra
  $A$ the lifted skew-closed structure on $\hmodM[A]$ is automatically right normal. The interested reader may now wonder what happens with left and associative normality. Since it will be easier to answer after discussing tricocycloids and tensor-hom adjunctions in \cref{ssec:tricocycloid} and \cref{ssec:tensorhom}, we postpone answering them until \cref{thm:assnorm} and \cref{prop:leftnorm}.

\subsection{Gabi-algebras and tricocycloids}\label{ssec:tricocycloid}
  Let us now show that a gabi-algebra $A$ canonically admits the structure of a so-called
  \emph{lax tricocycloid} in $\hmodM[k]$.
	This allows us also to discuss a few conditions under which a gabi-algebra structure on an algebra $A$ is coming from a Hopf algebra structure (see \cref{prop:gabi_algebra_tricocycloid_sufficient_hopf}).

  \begin{definition}[{\cite[\S3]{Street-tricocycloids}}]
    Let $(\cat,\otimes,\tensUnit)$ be a braided monoidal category with braiding $c$.
    A \emph{lax tricocycloid} in $\cat$ is an object $A$ together with a morphism $v
    \colon A \tensor A \to A \tensor A$ satisfying
    \begin{align}
      \label{eq:tricocycloid_equation}
      (v \tensor A) \circ (A \tensor c_{A, A}) \circ (v \tensor A)
      = (A \tensor v) \circ (v \tensor A) \circ (A \tensor v)
      \ .
    \end{align}
    The morphism $c_{A,A}^{-1} \circ v$ is a \emph{fusion operator} (see \cite[Proposition 1.1]{Street-tricocycloids}). When $v$ is invertible, we
    drop the adjective `lax'.
    A (lax) tricocycloid $(A, v)$ is \emph{augmented} if there are morphism $\eta \colon
    \tensUnit \to A$ and $\counit \colon A \to \tensUnit$ such that
    \begin{align}
      \label{eq:augmented_tricocycloid_equations}
      (A \tensor \counit) \circ v \circ (A \tensor \eta) = A, \hspace{6pt}
      (\counit \tensor A) \circ v = A \tensor \counit, \hspace{6pt}
      v \circ (\eta \tensor A) = A \tensor \eta, \hspace{6pt}
      \counit \circ \eta = \tensUnit\ .
    \end{align}
  \end{definition}

  The notion of tricocycloid encompasses both bialgebras and Hopf algebras in $\cat$.
  In fact, augmented lax tricocycloid structures $(v, \eta, \counit)$ on an object $A$
  are in one-to-one correspondence with bialgebra structures $(m, \Delta, \eta, \counit)$
  on $A$ with the inverse braiding, and this bialgebra is a Hopf algebra if and only if $v$ is
  invertible \cite[Proposition 2.3]{Lack-Street-skew-monoidales}.
  The tricocycloid equation \eqref{eq:tricocycloid_equation} encodes associativity,
  coassociativity, and the fact that $\Delta$ is a multiplicative map with respect to the
  inverse braiding.

  \begin{example}
    \label{ex:hopf_tricocycloid_1}
    Let $H$ be a Hopf algebra over $k$.
    Then $v(a \tensor b) = a\sweedler{2} b \tensor a\sweedler{1}$ and $v'(a \tensor b) =
    b\sweedler{1} \tensor S(b\sweedler{2}) a$ define lax tricocycloid structures.
    In fact, $v' = v\inv$, and $(A,v,1,\counit)$ is an augmented tricocycloid.
  \end{example}

  But we do not need a Hopf algebra or a bialgebra structure in order to get a tricocycloid.

  \begin{proposition}
    \label{prop:ga-bialgebras_are_tricocycloids}
    Let $A \in \hmodM[k]$ be a gabi-algebra and set $v(a \tensor b) \coloneqq b_+ \tensor b_- a$ for all $a,b \in A$.
    Then $(A, v)$ is a lax tricocycloid.
  \end{proposition}
  \begin{proof}
    We compute
    \begin{align*}
      ((v \tensor A) \circ (A \tensor \mathsf{tw}_{A, A}) \circ  (v \tensor A))(a \tensor b \tensor c)
      & \stackrel{\phantom{\eqref{eq:gabi_algebra_associativity}}}{=} ((v \tensor A) \circ (A \tensor \mathsf{tw}_{A, A}))(b_+ \tensor b_- a \tensor c) \\
      & \stackrel{\phantom{\eqref{eq:gabi_algebra_associativity}}}{=} (v \tensor A)(b_+ \tensor c \tensor b_- a) \\
      & \stackrel{\phantom{\eqref{eq:gabi_algebra_associativity}}}{=} c_+ \tensor c_- b_+ \tensor b_- a
      \ ,
    \end{align*}
    where $\mathsf{tw}$ is the usual twist, and on the other hand,
    \begin{align*}
      ((A \tensor v) \circ (v \tensor A) \circ (A \tensor v))(a \tensor b \tensor c)
      & \stackrel{\phantom{\eqref{eq:gabi_algebra_associativity}}}{=} ((A \tensor v) \circ (v \tensor A))(a \tensor c_+ \tensor c_- b) \\
      & \stackrel{\phantom{\eqref{eq:gabi_algebra_associativity}}}{=} (A \tensor v)(c_{++} \tensor c_{+-} a \tensor c_- b) \\
      & \stackrel{\phantom{\eqref{eq:gabi_algebra_associativity}}}{=} c_{++} \tensor (c_- b)_+ \tensor (c_- b)_- c_{+-} a \\
      & \stackrel{\phantom{\eqref{eq:gabi_algebra_associativity}}}{=} c_{++} \tensor c_{-+} b_+ \tensor b_- c_{--} c_{+-} a \\
      & \stackrel{\eqref{eq:gabi_algebra_associativity}}{=}
      c_{+} \tensor c_{-} b_+ \tensor b_- a \ .
    \end{align*}
    The expressions agree, and hence the claim follows. 
  \end{proof}

\subsection{Tensor-hom adjunction}\label{ssec:tensorhom}
  As it turns out, gabi-algebras allow for a `tensor-hom'-like adjunction on their
  categories of modules.

  Let $A$ be a gabi-algebra and $M$ be a left $A$-module.
  Consider the $k$-module $A \tensor M$, which we turn into an $A$-bimodule via
  \[a . (b \tensor m) . c = a b c_+ \tensor c_- m
    \ .
  \]
  We denote this bimodule by $A \odot M$.

  \begin{theorem}
    \label{prop:gabi_algebra_tensor_hom}
    For any $M \in \hmodM[A]$, there is an adjunction
    \begin{equation*}
      \begin{tikzcd}
        \hmodM[A] \ar[rr, bend left=20, "{(A \odot M) \tensor_A -}"]
        & \perp
        & \hmodM[A] \ar[ll, bend left=20, "{\hmodM[k](M, -)}"]
      \end{tikzcd}
      \ .
    \end{equation*}
    The unit and counit of the adjunction are
    \begin{align*}
      \coev^M_N \colon N \to \hmodM[k](M, (A \odot M) \tensor_A N),
      \quad n \mapsto (m \mapsto (1_A \odot m) \tensor_A n)
    \end{align*}
    and
    \begin{align*}
      \ev^M_N \colon (A \odot M) \tensor_A \hmodM[k](M, N) \to N,
      \quad (a \odot m) \tensor_A f \mapsto a f(m)
      \ .
    \end{align*}
  \end{theorem}

  It is routine to check the unit-counit adjunction.
  Instead of giving the proof, we will briefly explain where the adjunction comes from.
  First of all, the free-forgetful adjunction for $\hmodM[A]$ gives an isomorphism
  $\hmodM[k](M, N) \cong \hmodM[A](A \tensor M, N)$ of $k$-modules, for any $k$-module
  $M$ and $A$-module $N$.
  Transporting the left $A$-module structure on $\hmodM[k](M, N)$ through this
  isomorphism, one arrives at the right module structure on $A \odot M$.
  Finally, one uses the standard tensor-hom adjunction for bimodules over rings, see
  \eqref{eq:tensor_hom_adjunction}.
  The chain of isomorphisms is thus
  \begin{align*}
    \hmodM[A](N, \hmodM[k](M, P))
    \cong \hmodM[A](N, \hmodM[A](A \odot M, P))
    \cong \hmodM[A]((A \odot M) \tensor_A N, P)
    \ .
  \end{align*}

  \begin{remark}
    \Cref{prop:gabi_algebra_tensor_hom} actually holds more generally:
    this adjunction exists whenever $A$ is an algebra admitting an algebra map $\delta
    \colon A \to A \tensor A\op$.
  \end{remark}

\subsection{Skew-monoidal structure}\label{ssec:boxtimes}
  Any skew-closed category in which the inner hom functor admits a left adjoint possesses a
  skew-monoidal structure (see \cref{sec:skew_monoidal_cats}).
  Here we give the skew-monoidal structure of the category of modules over a
  gabi-algebra, associated to the tensor-hom adjunction from
  \Cref{prop:gabi_algebra_tensor_hom}.

  Consider the bifunctor
  \begin{align*}
    \boxtimes \colon \cat \times \cat \to \cat,
    \quad
    M \boxtimes N = (A \odot N) \tensor_A M
    \ .
  \end{align*}

  \begin{proposition}
    \label{prop:skew-monoidal-structure-of-gabi}
    Let $A$ be a gabi-algebra.
    The $\boxtimes$ tensor product defined above provides a left skew-monoidal structure on $\hmodM[A]$ with unit $k$,
    unitors
    \begin{gather}
      \lambda_N \colon k \boxtimes N \to N,
      \quad (a \odot n) \otimes_A 1_k \mapsto a n , \label{eq:leftunitor} \\
      \rho_M \colon M \xrightarrow{\sim} M \boxtimes k,
      \quad m \mapsto (1_A \odot 1_k) \otimes_A m, \notag
    \end{gather}
    and associator
		\begin{equation}\label{eq:reasso}
    \begin{gathered}
      \alpha_{L, M, N} \colon (L \boxtimes M) \boxtimes N
      \to L \boxtimes (M \boxtimes N), \\
      (a \odot n) \otimes_A \big((b \odot m) \otimes_A l \big)
      \mapsto
      \Big(a b_+ \odot \big((1_A \odot b_- n) \otimes_A m\big)\Big) \otimes_A l
      \ .
    \end{gathered}
		\end{equation}
	Moreover, this monoidal structure is strong (i.e. the unitors and associator are isomorhpisms) if and only if $A$ is a normal gabi-algebra.	
  \end{proposition}

  \begin{proof}
    The adjunction from \Cref{prop:gabi_algebra_tensor_hom} is given in terms of hom-sets
    as
    \begin{align*}
      \hmodM[A](X \boxtimes Y, V) \xrightarrow{\sim} \hmodM[A](X, \hmodM[k](Y, V)), 
      \quad
      f \mapsto \Big(f^\# \colon x \mapsto \big(y \mapsto f((1 \odot y) \tensor_A x)\big)\Big)
    \end{align*}
    with inverse
    \begin{align*}
      \hmodM[A](X, \hmodM[k](Y, V)) \xrightarrow{\sim} \hmodM[A](X \boxtimes Y, V) 
      \quad
      g \mapsto \big(g^\flat \colon (a \odot y) \tensor_A x \mapsto a g(x)(y)\big)
      \ ,
    \end{align*}
    where the action of $a$ is in $V$. We leave to the interested reader to make explicit each time to which adjunction $(-)^\#$ and $(-)^\flat$ refer.

    From the paragraph above \eqref{eq:unitors_from_adjunction}, we can now immediately describe the left unitor
    as
    \begin{align*}
      \lambda_X \colon k \boxtimes X = (A \odot X) \tensor_A k \to X, \quad
      \lambda_X \big((a \odot x) \tensor_A 1_k\big) 
      = j_X^\flat \big((a \odot x) \tensor_A 1_k\big)
      = a x,
    \end{align*}
    while unpacking \eqref{eq:unitors_from_adjunction} entails that the right unitor has to be
		\[\rho_X \colon X \to X \boxtimes k = (A \odot k) \tensor_A X, \qquad x \mapsto (1_A \odot 1_k) \tensor_A x.\]
		Notice that \eqref{eq:unitors_from_adjunction} is an isomorphism, since $\hmodM[k](k,X) \cong X$ as left $A$-modules, and $\rho_X = \id_{X \boxtimes k}^\#$. Furthermore, $\rho_X$ is invertible, as expected, with inverse 
    \begin{align*}
      \rho_X^{-1} \colon X \boxtimes k \to X, \quad
      \rho_X^{-1} \big((a \odot 1_k) \tensor_A x\big)
      = \id_X^\flat \big((a \odot 1_k) \tensor_A x\big)
      = a x
    \end{align*}
    (the invertibility of $\rho_X$ follows from \Cref{prop:bijection_skew_closed_monoidal_structures} and the fact that $\hmodM[A]$
    is right normal).

    Now, we compute the associator.
    First of all, the natural transformation from
    \eqref{eq:associativity_between_inner_hom_and_tensor} is given by $\left( p_{X, Y, Z}
    (h) \right) (x)(y) = h(1 \odot y \tensor_A x)$, for any linear map $h \colon X \otimes
    Y \to Z$.
    The natural transformation \eqref{eq:nat_transformation-associator} between hom 
    functors can be described as $f \mapsto (p \circ f^\#)^{\flat \flat}$, for an $A$-linear morphism
    $f\colon X \boxtimes (Y \boxtimes Z) \to V$.
    Explicitly, and by omitting a few $\boxtimes$,
    \begin{align*}
      (p \circ f^\#)^{\flat \flat}
      (a \odot z \tensor_A (b \odot y \tensor_A x))
      &= (p_{X, Y, X (YZ)} \circ f^\#)^{\flat \flat}
      (a \odot z \tensor_A (b \odot y \tensor_A x))
      \\ &= 
      a (p_{X, Y, X (YZ)} \circ f^\#)^{\flat}
      (b \odot y \tensor_A x) (z)
      \\ &\oversetEq[(\star)]
      a \left[ b . (p_{X, Y, X (YZ)} \circ f^\#) (x) (y) \right] (z)
      \\ &\oversetEq
      a b_+ (p_{X, Y, X (YZ)} \circ f^\#) (x) (y) (b_- z)
      \\ &\oversetEq
      a b_+ f^\# (x) (1 \odot b_- z \tensor_A y)
      \\ &\oversetEq
      a b_+ f (1 \odot (1 \odot b_- z \tensor_A y) \tensor_A x)
      \\ &\oversetEq
      f( a b_+ \odot (1 \odot b_- z \tensor_A y) \tensor_A x )
      \ ,
    \end{align*}
    where in the step marked $(\star)$ care has to be taken:
    the action of $b$ takes place in $\hmodM[k](Z, V)$.
    By using the Yoneda lemma, we finally find $\alpha_{X, Y, Z} = (p_{X, Y, X (YZ)} \circ
    \id_{X(YZ)}^\#)^{\flat \flat}$ and it is clear that this coincides with the associator
    in the statement.
    
    The last statement is a direct consequence of \cref{prop:bijection_skew_closed_monoidal_structures} and the definition of a normal gabi-algebra.
  \end{proof}

  Note also that
	\[
	k \boxtimes X = (A \odot X) \otimes_A k \cong (A \odot X) \otimes_A A/A^+ \cong (A \odot X)/(A \odot X)A^+,
	\] 
	where $A^+ \coloneqq \ker(\counit)$ is the augmentation ideal of $A$.

	\subsection{A family of non-trivial examples: one-sided Hopf algebras}\label{ssec:rightHopf}
		
		We already saw in \cref{ex:onesidedHopf} that right Hopf algebras with
    anti-mul\-ti\-plicative and anti-comultiplicative one-sided antipode are examples of gabi-algebras, whence their categories of modules are left skew-closed.
    It is interesting to remark explicitly that these provide an example
    of left skew-closed categories which are neither left normal nor associative normal.
		Let $(B,\Delta,\counit)$ be a $k$-bialgebra admitting an anti-multiplicative and anti-comultiplicative right antipode $S$ which is not a left antipode. Then $B$ is a left gabi-algebra with respect to $\counit$ and $\delta(b) \coloneqq b\sweedler{1} \otimes S(b\sweedler{2})$ for all $b \in B$, whence the category $\hmodM[B]$ is a left skew-closed category with respect to $[M,N] \coloneqq \hmodM[k](M,N)$ with action $(b \cdot f)(m) \coloneqq b\sweedler{1}f\big(S(b\sweedler{2})m\big)$ for all $b \in B$, $m \in M$, $f \in \hmodM[k](M,N)$ and for all $M,N$ in $\hmodM[B]$. Now, consider the $k$-linear map
		\begin{equation}\label{eq:betacan}
		\beta \colon B \otimes B \to B \otimes B, \qquad a \otimes b \mapsto a\sweedler{1} \otimes S(a\sweedler{2})b,
		\end{equation}
		and endow the domain with the diagonal $B$-module structure $a \cdot (b \otimes c) \coloneqq a\sweedler{1}b \otimes a\sweedler{2}c$ and the codomain with the left regular $B$-module structure $a \cdot (b \otimes c) \coloneqq ab \otimes c$ for all $a,b,c\in B$. It is clear that $\beta$ is not $B$-linear, otherwise the relation
		\[a\sweedler{1} \otimes S(a\sweedler{2})a\sweedler{3} = \beta(a \cdot (1 \otimes 1)) = a\cdot \beta(1 \otimes 1) = a \otimes 1\]
		would entail that $S$ is also a left antipode, contradicting our choice. Nevertheless,
		\begin{align*}
		a\sweedler{1}\cdot \beta\left(S\left(a\sweedler{2}\right)\cdot (b \otimes c)\right) & = a\sweedler{1}\cdot \beta\left(S\left(a\sweedler{3}\right)b \otimes S\left(a\sweedler{2}\right)c\right) \\
		& = a\sweedler{1}\cdot \left(S\left(a\sweedler{3}\right)\sweedler{1}b\sweedler{1} \otimes S\left(S\left(a\sweedler{3}\right)\sweedler{2}b\sweedler{2}\right)S\left(a\sweedler{2}\right)c\right) \\
		& = a\sweedler{1}S\left(a\sweedler{4}\right)b\sweedler{1} \otimes S\left(a\sweedler{2}S\left(a\sweedler{3}\right)b\sweedler{2}\right)c = \counit(a)\beta(b \otimes c)
		\end{align*}
		implies that the assignment $1_k \mapsto \beta$ is an element in $\hmodM[B]\left(k,[B \otimes B, B \otimes B]\right)$ which does not come from an element in $\hmodM[B](B \otimes B, B\otimes B)$ via the morphism in \ref{item:N1}. Therefore, the skew-closed structure is not left normal. It is not associative normal either, or \cref{thm:assnorm} would imply that $\beta$ is invertible, which is not the case.

		As a final observation remember that, since $\hmodM[B]$ is not left normal as a skew-closed category, \cref{prop:bijection_skew_closed_monoidal_structures} entails that the unitor $\lambda_M$ from \eqref{eq:assunitors} cannot be a natural isomorphism. It is interesting to see explicitly why this is the case in our concrete example.
		To this aim, notice that even if $\beta$ from \eqref{eq:betacan} is not $B$-linear, it induces a left $B$-linear map
		\[
		\hat\beta\colon \frac{B \odot ({.B} \otimes {.B})}{(B \odot ({.B} \otimes {.B}))B^+} \to \frac{B \odot ({.B} \otimes B)}{(B \odot ({.B} \otimes B))B^+}, \qquad \overline{a \otimes (b\otimes c)} \mapsto \overline{a \otimes (b\sweedler{1} \otimes S(b\sweedler{2})c)},
		\]
		because for all $x \in B^+$ we have
		\begin{gather*}
		\overline{ax\sweedler{1} \otimes (S(x\sweedler{3})\sweedler{1}b\sweedler{1} \otimes S\big(S(x\sweedler{3})\sweedler{2}b\sweedler{2}\big)S(x\sweedler{2})c)} = \overline{ax\sweedler{1} \otimes (S(x\sweedler{2})b\sweedler{1} \otimes S(b\sweedler{2})c)} \\
		= \varepsilon(x) \overline{a \otimes (b\sweedler{1} \otimes S(b\sweedler{2})c)} = 0.
		\end{gather*}
		If, in this framework, 
		\[\lambda_{{.B} \otimes {.B}} \colon \frac{B \odot ({.B} \otimes {.B})}{(B \odot ({.B} \otimes {.B}))B^+} \to {.B} \otimes {.B}, \qquad \overline{a \otimes (b\otimes c)} \mapsto a\sweedler{1}b \otimes a\sweedler{2}c,\]
		is invertible, with inverse necessarily given by
		\[\lambda_{{.B} \otimes {.B}}^{-1} \colon {.B} \otimes {.B} \to \frac{B \odot ({.B} \otimes {.B})}{(B \odot ({.B} \otimes {.B}))B^+}, \qquad a \otimes b \mapsto \overline{a\sweedler{1} \otimes (1 \otimes S(a\sweedler{2})b)},\]
		then we have the equality
		\[\overline{a \otimes (1 \otimes 1)} = \overline{a\sweedler{1} \otimes (1 \otimes S(a\sweedler{2})a\sweedler{3})}\]
		in $B \odot ({.B} \otimes {.B})/(B \odot ({.B} \otimes {.B}))B^+$. By applying $\hat\beta$, the latter is an equality in $B \odot ({.B} \otimes {B})/(B \odot ({.B} \otimes {B}))B^+$ as well, but now we may apply $\counit \otimes \counit \otimes \id$ to it (it is well-defined) to conclude that $\counit(a) = S(a\sweedler{1})a\sweedler{2}$. A contradiction.
 
		Modules over one-sided Hopf algebras also offer an example of a monoidal and left skew-closed category whose structures are not compatible, that is, which does not form a left skew-closed skew-monoidal category in the sense of \cref{def:left_skew_closed_skew_monoidal}. Indeed, if $B$ is as above, then $(\hmodM[B],\otimes,k)$ is a monoidal category and $(\hmodM[B],\hmodM[k](-,-),k)$ is a left skew-closed category, but the functor $- \otimes M$ is not left adjoint to the functor $\hmodM[k](M,-)$: while the unit 
		\[\coev^{M}_N \colon N \to \hmodM[k](M, N \otimes M), \qquad n \mapsto (m \mapsto n \otimes m)\]
		is left $B$-linear, the unit
		\[\ev^{M}_N \colon \hmodM[k](M, N) \otimes M \to N,\qquad f \otimes m \mapsto f(m)\]
    is not, unless $S$ is a two-sided antipode.

	\section{When are gabi-algebras Hopf algebras?}\label{sec:gabi-Hopf}
	
	We saw (cf.\ \cref{ex:hopf_algebras_are_gabi}) that Hopf algebras are examples of gabi-algebras, but also that not any gabi-algebra is Hopf (see \cref{ex:onesidedHopf}). This section is devoted to answer the question when the gabi and Hopf notions coincide.

\subsection{Finite-dimensional double gabi-algebras are Hopf}\label{ssec:GabiHopf}
	
	  We first establish an easy criterion for tricocycloids to give Hopf algebras.

  \begin{lemma}
    \label{prop:inverse_lax_tricocycloid}
    Let $(A, v)$ be a lax tricocycloid in a symmetric monoidal category.
    \begin{enumerate}[label=(\arabic*),ref=(\arabic*),leftmargin=0.8cm]
      \item\label{item:tri1}
        If $v$ is invertible, then $(A, v\inv)$ is a lax tricocycloid as well.

      \item\label{item:tri2}
        Assume that there are morphisms $\eta \colon \tensUnit \to A$ and $\counit \colon
        A \to \tensUnit$ such that
        \begin{align*}
          \counit \tensor A = (A \tensor \counit) \circ v, \quad
          \eta \tensor A = v \circ (A \tensor \eta) , \quad
          \counit \circ \eta = 1
          \ .
        \end{align*}
        If $v$ is invertible and $(A \tensor \counit) \circ v\inv \circ (A \tensor \eta) = A$
        holds, then $(A, v\inv, \eta, \counit)$ is an augmented tricocycloid, and thus
        equips $A$ with the structure of a Hopf algebra.
    \end{enumerate}
  \end{lemma}

  \begin{proof}
    \ref{item:tri1} follows directly from a simple computation, using that the braiding
    is its own inverse.
    \ref{item:tri2} is easy as well:
    if $v$ is invertible and satisfies the given equations, then $v\inv$ satisfies
    \begin{align*}
      (\counit \tensor A) \circ v\inv = A \tensor \counit, \quad
      v\inv \circ (\eta \tensor A) = A \tensor \eta
      \ .
    \end{align*}
    Together with $\counit \circ \eta = 1$ and the assumption $(A \tensor \counit) \circ
    v\inv \circ (A \tensor \eta) = A$, these are exactly the defining equations
    \eqref{eq:augmented_tricocycloid_equations} making the tricocycloid $(A, v\inv)$
    augmented.
  \end{proof}

  \begin{example}
    \label{ex:hopf_tricocycloid_2}
    Continuing \Cref{ex:hopf_tricocycloid_1}, $(A, v', 1, \counit)$ satisfies the
    assumptions of the second point of \Cref{prop:inverse_lax_tricocycloid}.
  \end{example}
	
	Define now for a gabi-algebra $A$ the \emph{canonical map} 
	\begin{equation}\label{eq:can}
	\beta\colon A \otimes A \to A \otimes A, \qquad a \tensor b \mapsto a_+ \tensor a_- b\ .
	\end{equation}
  If $\beta$ is invertible, we set $\Delta(a) = \beta\inv(a \tensor 1)$.
  We can now show that gabi-algebras with invertible canonical map are quite close to
  Hopf algebras. This will also be helpful in proving \cref{prop:leftnorm}.

  \begin{theorem}
    \label{prop:gabi_algebra_tricocycloid_sufficient_hopf}
    Let $A$ be a gabi-algebra.
    If $\beta$ is invertible and $\Delta$ is left counital, i.e.\ $(\counit \tensor A)
    \circ \Delta = A$, then $(A, \Delta, \counit)$ is a coalgebra.
    In fact, $A$ is a Hopf algebra whose antipode is the antipode $\sigma$ of $A$.
  \end{theorem}

  \begin{proof}
    Let $\mathsf{tw}$ be the flip map in $\hmodM[k]$, i.e.\ the canonical symmetric braiding.
    By \Cref{prop:ga-bialgebras_are_tricocycloids}, $(A, v = \beta \circ \mathsf{tw})$ is a lax
    tricocycloid.
    Clearly, $v$ is invertible if and only if $\beta$ is.
    We also have $\counit \circ \eta = 1$, and one computes
    \begin{align*}
      ((A \tensor \counit) \circ v)(a  \tensor  b)
      = b_+ \counit(b_- a)
      = (\counit  \tensor  A)(a  \tensor  b)
      \ ,
    \end{align*}
    and
    \begin{align*}
      (v \circ (A \tensor \eta))(a)
      = 1_+  \tensor  1_- a
      = 1  \tensor  1 a
      = (\eta  \tensor  A)(a)
      \ .
    \end{align*}
    Finally, we have
    \begin{align*}
      (A  \tensor  \counit) \circ v\inv \circ (A  \tensor  \eta) = A
      \iff
      (\counit  \tensor  A) \circ \beta\inv \circ (A  \tensor  \eta) = A
      \ ,
    \end{align*}
    and thus all the assumptions for the second point of
    \Cref{prop:inverse_lax_tricocycloid} are satisfied.
    This shows that we indeed get a Hopf algebra.
    By the general theory of augmented lax tricocycloids with invertible structure map,
    the antipode is given by
    \begin{align*}
      S(a) 
      &= ((\counit \tensor A) \circ v \circ (\eta \tensor A))(a)
      = ((\counit \tensor A) \circ \beta \circ (A \tensor \eta))(a)
      \\
      &= ((\counit \tensor A) \circ \beta)(a \tensor 1)
      = (\counit \tensor A)(a_+ \tensor a_- 1)
      = \counit(a_+) a_-
      = \sigma(a)
      \ ,
    \end{align*}
    as claimed.
  \end{proof}
	
The previous result allows us to reduce a first class of gabi-algebras to Hopf algebras.	
	
	\begin{corollary}\label{fdinvanti}
    Let $k$ be a field and $A$ be a finite-dimensional gabi-algebra over
    $k$ with invertible antipode $\sigma$.
    Define $\Delta(a) = a_+ \tensor \sigma\inv (a_-)$.
    Then $(A, m, 1, \Delta, \counit, \sigma)$ is a Hopf algebra.
  \end{corollary}

  \begin{proof}
    In $A$, we have the identity
    \begin{align*}
      a \tensor 1 
      ~~\oversetEq[{\eqref{eq:gabi_algebra_associativity}}]~~ 
      a_{++} \tensor \counit(a_{-+}) a_{--} a_{+-}
      = a_{++} \tensor \sigma(a_-) a_{+-}
      \ ,
    \end{align*}
    which upon applying $\sigma\inv$ to the second leg becomes
    $
      a \tensor 1 = a_{++} \tensor \sigma\inv(a_{+-}) a_-
      .
    $
    Thus the composition
    \begin{align*}
      a \tensor b 
      \overset{\beta}{\mapsto}
      a_+ \tensor a_- b
      \mapsto
      a_{++} \tensor \sigma\inv(a_{+-}) a_- b
    \end{align*}
    is the identity, i.e.\ the canonical map $\beta$ has a left inverse.
    By finite-dimensionality of $A$, $\beta$ is invertible.
    Note that $\Delta(a) = \beta\inv(a \tensor 1)$.
    If $\Delta$ is left counital, then by
    \Cref{prop:gabi_algebra_tricocycloid_sufficient_hopf} we are done.
    So we compute
    \begin{align*}
      ((\counit \tensor A) \circ \Delta)(a)
      = \counit(a_+) \sigma\inv(a_-)
      = \sigma\inv( \sigma(a) )
      = a
      \ ,
    \end{align*}
    exactly as needed.
  \end{proof}

    Our next aim is to state a sufficient condition for the antipode of a gabi-algebra to be
    invertible.
    First, note that for any \emph{closed monoidal} category $\cat$ with left and right
    internal homs $[-,-]^l$ and $[-,-]^r$, respectively, one has natural isomorphisms
    \begin{align*}
      \cat(X, [Y, Z]^r) \cong \cat(X \tensor Y, Z) \cong \cat(Y, [X, Z]^l)
      \ .
    \end{align*}
    One can now generalize this to the closed (non-monoidal) setting by simply
    demanding two closed structures with the same unit on the category such that an
    isomorphism between the two outer hom-sets hold.
    Let us apply this to the gabi-algebra setting.
    
    \begin{lemma}
    Let $(A,\counit)$ be an augmented algebra endowed with two gabi-algebra structures denoted by $\delta(a) = a_+ \tensor a_-$ and $\delta'(a)= a_{+'} \tensor a_{-'}$. Denote the associated closed structures on $\hmodM[A]$ by $[-,-]$ and $[-,-]'$ respectively.
    Then, the following are equivalent:
    \begin{enumerate}[label=(\arabic*),ref=(\arabic*),leftmargin=0.8cm]
      \item\label{item:5.5.1}
        The linear map
        \begin{align*}
          \phi \colon \hmodM[A](X, [Y, Z]) \to \hmodM[A](Y, [X, Z]')
          , \quad \phi(f)(y)(x) = f(x)(y)
        \end{align*}
        and its obvious inverse 
				\[\phi\inv\colon \hmodM[A](Y, [X, Z]') \to \hmodM[A](X, [Y, Z]), \quad \phi\inv(g)(x)(y) = g(y)(x)\] 
				are well-defined, entailing that $\hmodM[A](X, [Y, Z]) \cong \hmodM[A](Y, [X, Z]')$.

      \item\label{item:5.5.2}
        $a_{+'} a_{-'+} \tensor a_{-'-} = 1 \tensor a$ and $a_{+} a_{-+'} \tensor a_{--'} = 1 \tensor a$ hold for all $a \in A$.
    \end{enumerate}
    If one of these equivalent conditions holds, then we call $A$ a {\em double} gabi-algebra.
    \end{lemma}

    \begin{proof}
		We first show that $\phi$ is well defined if and only if the left-hand side relation in \ref{item:5.5.2} holds. For all $f$ in $\hmodM[A](X, [Y, Z])$, $\phi(f)$ is $A$-linear if and only if the left-most term in
		\begin{align*}
		\big(a.\phi(f)(y)\big)(x) & = a_{+'}\phi(f)(y)(a_{-'}x) = a_{+'}f(a_{-'}x)(y) = a_{+'}\big(a_{-'}.f(x)\big)(y) \\
		& = a_{+'}a_{-'+}f(x)(a_{-'-}y)
		\end{align*}
equals
\[\phi(f)(ay)(x) = f(x)(ay)\]
for all $x \in X$, $y \in Y$, $a \in A$. Clearly, if $a_{+'} a_{-'+} \tensor a_{-'-} = 1 \tensor a$ holds for all $a \in A$, then this is the case. Conversely, assume that $\phi$ in \ref{item:5.5.1} is well-defined. Let $X = Y = {}_AA$ and $Z = {}_A A \tensor A$ and consider $f \in \hmodM[A](A, [A, A \ot A])$ uniquely determined by
    $f(1)(a) \coloneqq 1 \tensor a$ for all $a \in A$.
    Then we have
    \begin{align*}
      1 \tensor a & = f(1)(a) = \phi(f)(a)(1) = \big(a.\phi(f)(1)\big)(1) = a_{+'} \phi(f)(1)(a_{-'}) 
      = a_{+'} f(a_{-'})(1) \\
			& = a_{+'} \big(a_{-'}.f(1)\big)(1) = a_{+'}a_{-'+} f(1)(a_{-'-}) = a_{+'}a_{-'+} \ot a_{-'-}
    \end{align*}
    for all $a \in A$, as desired. A straightforward adaptation of this argument shows that $\phi\inv$ is well defined if and only if the right-hand side relation in \ref{item:5.5.2} holds. To conclude, it is clear that if $\phi$ and $\phi\inv$ are well-defined, then they are each other inverses.
  \end{proof}

    \begin{lemma}\label{doubleinvantipode}
    Let $A$ be a double gabi-algebra, then the antipode $\sigma$ is invertible.
\end{lemma}    
    
    \begin{proof}
    In this case, one has
    \begin{align*}
      \sigma(\sigma'(a)) = \counit(a_{+'} a_{-'+}) a_{-'-} = a
      \qandq
      \sigma'(\sigma(a)) = \counit(a_{+} a_{-+'}) a_{--'} = a
      \ ,
    \end{align*}
    so that the antipode is invertible.
\end{proof}    

Combining \cref{doubleinvantipode} with \cref{fdinvanti} we immediately obtain the following.
    
	\begin{corollary}
	A finite-dimensional double gabi-algebra is a Hopf algebra.
	\end{corollary}

	\subsection{Commutative gabi-algebras are Hopf algebras}

  Throughout this section, $A$ shall always be a gabi-algebra over the commutative ring
  $k$.
  
  \begin{proposition}
    \label{prop:commutative_gabi_are_hopf}
    If the gabi-algebra $A$ is commutative, then it is a Hopf algebra with
    comultiplication
    $
      \Delta(a) := a\sweedler{1} \tensor a\sweedler{2} := a_+ \tensor \sigma(a_-)
    $
    and antipode $\sigma$.
  \end{proposition}

  \begin{proof}
	Since $A$ is commutative, on the one hand we have that
	\[a_{++} \otimes a_{+-}\sigma(a_-) = a_{++} \otimes \counit(a_{-+})a_{+-}a_{--} = a_{++}\counit(a_{-+}) \otimes a_{--}a_{+-} \stackrel{\eqref{eq:gabi_algebra_associativity}}{=} a \otimes 1,\]
	and on the other the computation
	\[\sigma^2(a) = \counit(a_+)\sigma(a_-) \stackrel{\eqref{eq:gabi_algebra_antipode_axioms}}{=} a_{++}a_{+-}\counit(a_{-+})a_{--} = a_{++}\counit(a_{-+})a_{--}a_{+-} \stackrel{\eqref{eq:gabi_algebra_associativity}}{=} a\]
	entails that also
	\[a_{++} \otimes \sigma(a_{+-})a_- = a_{++} \otimes \sigma\big(\sigma(a_-)a_{+-}\big) = a \otimes 1\]
	for all $a \in A$, whence the canonical map $\beta$ from \eqref{eq:can} is invertible with inverse $\beta^{-1}(a \otimes b) \coloneqq  a_+ \tensor \sigma(a_-)b$. If we set $\Delta(a) = a_+ \otimes \sigma(a_-)$, then it is left counital because 
	\[(\counit \tensor A)(\Delta(a)) = \counit(a_+)\sigma(a_-) = \sigma^2(a) = a\]
	and so the statement follows from \cref{prop:gabi_algebra_tricocycloid_sufficient_hopf}.
  \end{proof}

  \begin{corollary}
    A commutative algebra lifts the closed structure of $\hmodM[k]$ if and only if it
    lifts the closed monoidal structure.
  \end{corollary}

	\subsection{Gabi-algebras and normality}
	
	We are now in the position of establishing which additional properties on $A$ correspond to associative and left normality of the skew-closed structure, as promised in \cref{ssec:defandex}.
	
	\begin{proposition}\label{thm:assnorm}
	Let $A$ be a gabi-algebra over $k$. Then the skew-closed structure on $\hmodM[A]$ is associative normal if and only if the canonical map $\beta$ from \eqref{eq:can} is invertible.
	\end{proposition}

	\begin{proof}
	Recall from \cref{prop:bijection_skew_closed_monoidal_structures} that the closed structure on $\hmodM[A]$ is associative normal if and only if every component of the associative constraint from \eqref{eq:reasso}, that is
	\[
    \begin{gathered}
      \alpha_{L, M, N} \colon (L \boxtimes M) \boxtimes N
      \to L \boxtimes (M \boxtimes N), \\
      (a \odot n) \otimes_A \big((b \odot m) \otimes_A l \big)
      \mapsto
      \big(a b_+ \odot \big((1_A \odot b_- n) \otimes_A m\big)\big) \otimes_A l
      \ ,
    \end{gathered}
		\]
		is an isomorphism. Therefore, we are going to show that $\alpha_{L,M,N}$ is an isomorphism for every $L,M,N$ in $\hmodM[A]$ if and only if $\beta$ admits an inverse, that we write $\beta^{-1}(a \otimes b) = a\sweedler{1} \otimes a\sweedler{2}b$ for all $a,b \in A$ by a slight abuse of notation that will be soon justified (see \cref{prop:leftnorm}). 
		In such a case,
		\begin{equation}\label{eq:antipode}
		\begin{gathered}
		{a\sweedler{1}}_+ \otimes {a\sweedler{1}}_-{a\sweedler{2}} = a \otimes 1 = {a_+}\sweedler{1} \otimes {a_+}\sweedler{2}a_- \\
		\text{and} \qquad a_+ \otimes {a_-}\sweedler{1} \otimes {a_-}\sweedler{2} = a_{++} \otimes a_- \otimes a_{+-}
		\end{gathered}
		\end{equation}
		for all $a \in A$ (the latter follows by applying $\id \otimes \beta$ to both sides and comparing the results). 
		Observe that since $\beta(a \otimes 1) = \delta(a) = a_+ \otimes a_-$ is unital and multiplicative, we have that
		\begin{align*}
		\beta(a\sweedler{1}b\sweedler{1} \otimes a\sweedler{2}b\sweedler{2}) & = (a\sweedler{1}b\sweedler{1})_+ \otimes (a\sweedler{1}b\sweedler{1})_-a\sweedler{2}b\sweedler{2} = {a\sweedler{1}}_+{b\sweedler{1}}_+ \otimes {b\sweedler{1}}_-{a\sweedler{1}}_-a\sweedler{2}b\sweedler{2} \\
		& = a{b\sweedler{1}}_+ \otimes {b\sweedler{1}}_-b\sweedler{2} = ab \otimes 1 = \beta\beta^{-1}(ab \otimes 1),
		\end{align*}
		whence $A \to A \otimes A, a \mapsto \beta^{-1}(a \otimes 1)$ is a morphism of algebras.
		
		To begin with, let us slightly simplify \eqref{eq:reasso}. Note that
		\[(L \boxtimes M) \boxtimes N = (A \odot N) \otimes_A (L \boxtimes M) = (A \odot N) \otimes_A \big((A \odot M) \otimes_A L\big) \cong  \big((A \odot N) \odot M\big) \otimes_A L\]
		via $(a \odot n) \otimes_A \big((b \odot m) \otimes_A l \big)
      \mapsto
      \big((a b_+ \odot b_- n) \odot m\big) \otimes_A l$ in one direction, for all $a,b \in A$, $m \in M$, $n \in N$, $l \in L$, and via $\big((a  \odot n) \odot m\big) \otimes_A l \mapsto (a \odot n) \otimes_A \big((1_A \odot m) \otimes_A l \big)$ in the other. Up to the latter isomorphism,
			\[
    \begin{aligned}
      \alpha_{L, M, N} \colon \big((A \odot N) \odot M\big) \otimes_A L
      & \to \big( A \odot \big((A \odot N) \otimes_A M\big)\big) \otimes_A L, \\
      \big((a \odot n) \odot m\big) \otimes_A l
      & \mapsto
      \big(a \odot \big((1_A \odot n) \otimes_A m\big)\big) \otimes_A l
      \ ,
    \end{aligned}
		\]
		
		Suppose now that $\beta$ is invertible and consider
		\[
    \begin{aligned}
       \big( A \odot \big((A \odot N) \otimes_A M\big)\big) \otimes_A L
      & \to \big((A \odot N) \odot M\big) \otimes_A L, \\
      \big(a \odot \big((b \odot n) \otimes_A m\big)\big) \otimes_A l
      & \mapsto
      \big((a \odot b\sweedler{2}n) \odot b\sweedler{1}m\big) \otimes_A l
      \ ,
    \end{aligned}
		\]
		which is well-defined because of \eqref{eq:antipode}. By a direct check, it turns out to be a two-sided inverse of $\alpha_{L,M,N}$, so one implication is proved.
		
		For the other implication, suppose that $\alpha$ is a natural isomorphism and consider the isomorphism
		\[
    \begin{aligned}
      (A \odot A) \odot A \cong \big((A \odot A) \odot A\big) \otimes_A A
      & \xrightarrow{\alpha_{A, A, A}} \big( A \odot \big((A \odot A) \otimes_A A\big)\big) \otimes_A A \cong A \odot (A \otimes A), \\
      (a \odot b) \odot c
      & \mapsto
      a \odot (c_+ \otimes c_-b)
      \ ,
    \end{aligned}
		\]
		where $A \otimes A$ is a left $A$-module with regular left $A$-action on the left-most tensorand. If we look at it as a $k$-linear morphism
		\[\alpha' \colon (A \otimes A) \otimes A \to A \otimes (A \otimes A), \qquad (a \otimes b) \otimes c \mapsto a \otimes (c_+ \otimes c_-b),\]
		its inverse is uniquely determined by 
		\[(x'\otimes x'') \otimes x''' \coloneqq {\alpha'}^{-1}\big(1 \otimes (x \otimes 1)\big)\]
		for all $x \in A$. The latter satisfies
		\[
		1 \otimes x \otimes 1 = x'\otimes {x'''}_+ \otimes {x'''}_-x'' \qquad \text{and} \qquad 1 \otimes 1 \otimes y = {y_+}'\otimes {y_+}''y_- \otimes {y_+}'''
		\]
		for all $x,y \in A$. In particular,
		\[x \otimes 1 = {x'''}_+ \otimes {x'''}_-\varepsilon(x')x''\]
		entails that $a \otimes b \mapsto a''' \otimes \varepsilon(a')a''b$ is a section of $\beta$, while 
		\[y \otimes 1 = {y_+}''' \otimes \varepsilon\big({y_+}'\big){y_+}''y_-\]
		entails that it is a retraction of $\beta$ as well, hence it is its inverse.
	\end{proof}
	
	\begin{proposition}\label{prop:leftnorm}
	Let $A$ be a gabi-algebra whose canonical morphism $\beta$ is invertible. Then the closed structure on $\hmodM[A]$ is left normal, too, if and only if $A$ is a Hopf algebra with comultiplication $\beta^{-1}(a \otimes 1)$, counit $\counit$ and antipode $\sigma(a) = \counit(a_+)a_-$ for all $a \in A$.
	\end{proposition}
	
	\begin{proof}
	Clearly, if the gabi-algebra structure on $A$ comes from a Hopf algebra structure, then the claim in the statement holds. Thus, let us focus on the other implication.
	Recall from \cref{prop:bijection_skew_closed_monoidal_structures} that the closed structure on $\hmodM[A]$ is left normal if and only if every component 
	\[
      \lambda_M \colon (A \odot M)/(A \odot M)A^+ \to M,
      \quad \overline{a \odot m} \mapsto a m ,
    \]
	 of the left unit constraint from \eqref{eq:leftunitor}	is an isomorphism. Since, by hypothesis, $\beta$ is invertible with inverse $\beta^{-1}(a \otimes b) = a\sweedler{1} \otimes a\sweedler{2}b$ for all $a,b \in A$, we know that
	\[\overline{a \odot m} = \overline{{a\sweedler{1}}_+ \odot {a\sweedler{1}}_-a\sweedler{2}m} = \overline{1_A \odot \counit(a\sweedler{1})a\sweedler{2}m}\]
	for all $a \in A$, $m \in M$. In particular, if $\lambda_A$ is invertible, then
	\[\overline{1_A \odot ab} = \overline{a \odot b} = \overline{1_A \odot \counit(a\sweedler{1})a\sweedler{2}b}\]
	and so $ab = \counit(a\sweedler{1})a\sweedler{2}b$ for all $a,b \in A$. The conclusion now follows from \cref{prop:gabi_algebra_tricocycloid_sufficient_hopf}.
	\end{proof}

We are now ready to formulate the main theorem of this paper, the proof of which now follows directly from the result above.

\begin{theorem}\label{cor:urka}
Let $A$ be an algebra. Then there is a bijective correspondence between normal gabi-algebra structures on $A$ and Hopf algebra structures on $A$.

In other words, $\hmodM[A]$ is a normal closed category with strictly closed forgetful functor $\hmodM[A] \to \hmodM[k]$ if and only if $\hmodM[A]$ is closed monoidal and the forgetful functor $\hmodM[A] \to\hmodM[k]$ is strictly closed monoidal.
\end{theorem}

\begin{remark}
One could wonder under which conditions the last theorem can be extended to the more general setting of algebras in a closed monoidal category (possibly different from $\hmodM[k]$). To this end, let us make the following observation, which allows us to recover the above result in such a more general setting. Consider two (strong) closed monoidal categories $\mathcal C$ and $\mathcal D$, and let $F:\mathcal C\to\mathcal D$ be an (essentially) surjective-on-objects, strict closed functor. Then we have the following natural isomorphisms:
\begin{align*}
& [F(X\ot Y), FZ] =  F[X\ot Y,Z] \cong F[X,[Y,Z]] \\
 & = [FX,F[Y,Z]] = [FX,[FY,FZ]] \cong [FX\ot FY,FZ],
\end{align*}
where the equalities follow from the fact that the functor $F$ is strict closed, and the isomorphisms follow from the closed monoidal structures of $\mathcal C$ and $\mathcal D$. Since $F$ is surjective on objects and by invoking enriched Yoneda (see, e.g., \cite[\S1.9]{Kelly}), we obtain a natural isomorphism $F(X\ot Y)\cong  FX\ot FY$, so $F$ is also strong monoidal. 

Then, consider a normal gabialgebra $A$ in a closed monoidal category $\mathcal D$, let $\mathcal C$ be the closed monoidal category of $A$-modules (see \cref{prop:skew-monoidal-structure-of-gabi})  and $F:\hmodM[A]\to \mathcal D$ the forgetful functor. We showed that modules are skew closed monoidal, and normality of $A$ means that it is closed monoidal. Moreover $F$ is strict closed by construction and essentially surjective on objects since thanks to the augmentation map $\counit$, we can endow any object of $\mathcal D$ with a trivial $A$-module structure. Then by the above $F$ is strong monoidal and therefore usual Tannaka-Krein duality implies that $A$ is Hopf.
\end{remark}

% -------------------------------------------------------------------- %
% ACKNOWLEDGEMENTS
% -------------------------------------------------------------------- %

\addtocontents{toc}{\SkipTocEntry}
\section*{Acknowledgements}  

Johannes Berger worked as a postdoctoral researcher within the framework of the ARC project ``From algebras to combinatorics, and back'' while collaborating on this project.

Paolo Saracco is a Charg\'e de Recherches of the Fonds de la Recherche Scientifique - FNRS and a member of the ``National Group for Algebraic and Geometric Structures and their Applications'' (GNSAGA-INdAM).

Joost Vercruysse thanks the Fédération Wallonie-Bruxelles for support via the aforementioned ARC project.

The authors thank Gabriella B\"ohm for the interesting discussions and for her contribution \cite{Boehm-private} to the project.

%%%%%%%%%%%%% BIBLIOGRAPHY %%%%%%%%%%%%%%%%%%%%

\pagebreak

\appendix

\section{Mixed liftings}\label{appendix}
  Consider categories $\cat[D], \cat[E]$, and $\cat$.
  On them, consider a comonad $(V, \Delta, \counit)$ and monads $(W, m^W, u^W)$ and $(T,
  m^T, u^T)$, respectively.
  Now let $F \colon \cat[D] \times \cat[E] \to \cat$ be a functor.
  The question is:
  when does this lift to a functor $\cat[D]^V \times \cat[E]^W \to \cat^T$?

  To answer it, let us introduce the following notion.

  \begin{definition}
    A natural transformation of the form $\nu_{X, Y} \colon TF(VX, Y) \to F(X, WY)$ is
    said to satisfy the \emph{lifting property of $F$ with respect to $(V, W, T)$} if the
    the following hold:
    \begin{itemize}[leftmargin=0.8cm]
      \item
        \emph{mixed (co)unitality:} 
        \begin{equation*}
          \begin{tikzcd}
            {F(VX, Y)} 
            \ar[r, "{u^T_{F(VX, Y)}}"] 
            \ar[dr, swap, "{F(\counit_X, u^W_Y)}"]
            & {TF(VX, Y)} \ar[d, "\nu_{X, Y}"]
            \\
            & {F(X, WY)}
          \end{tikzcd}
        \end{equation*}
        and
      \item
        \emph{mixed (co)multiplicativity:}
        \begin{equation*}
          \begin{tikzcd}[column sep=large]
            {T^2F(VX, Y)} \ar[rrr, "{m^T_{F(VX, Y)}}"] 
            \ar[d, swap, "{T^2 F(\Delta_X, Y)}"]
            & & & {TF(VX, Y)} \ar[d, "\nu_{X, Y}"]
            \\
            {T^2 F(V^2 X, Y)} \ar[r, swap, "{T \nu_{VX, Y}}"]
            & {TF(VX, WY)} \ar[r, swap, "\nu_{X, WY}"]
            & {F(X, W^2 Y)} \ar[r, swap, "{F(X, m^W_Y)}"]
            & {F(X, WY)}
          \end{tikzcd}
          \ .
        \end{equation*}
    \end{itemize}
  \end{definition}

  Similar to the usual lifting theorem for monads, we can use these to classify liftings
  in our current situation.

  \begin{theorem}
    \label{prop:mixed_liftings}
    Let $\cat, \cat[D], \cat[E], V, W, T,$ and $F$ be as above.
    The liftings of $F$ are in one-to-one correspondence with natural transformations
    satisfying the lifting property of $F$ w.r.t.\ $(V, W, T)$.

    More precisely, if $\nu$ satisfies the lifting property, and $((X, x), (Y, y)) \in
    \cat[D]^V \times \cat[E]^W$, then $F(X, Y)$ is equipped with the $T$-action 
    \begin{align*}
      \triangleright^\nu_{x, y}
      = TF(X, Y) \xrightarrow{TF(x, Y)}
      TF(VX, Y) \xrightarrow{\nu_{X, Y}}
      F(X, WY) \xrightarrow{F(X, y)}
      F(X, Y)
      \ .
    \end{align*}
    The action on morphisms is determined by $F$.
    Conversely, if the functor lifts, i.e.\ every pair of (co)actions $x, y$ lifts to an
    action $\triangleright_{x, y} \colon TF(X, Y) \to F(X, Y)$, then the natural
    transformation $\nu^{\triangleright}_{X, Y}$ satisfying the lifting property is
    obtained as the composition
    \begin{align*}
      TF(VX, Y)
      \xrightarrow{TF(VX, u^W_Y)} TF(VX, WY)
      & \xrightarrow{\triangleright_{\Delta_X, m^W_Y}} F(VX, WY)
      \xrightarrow{F(\counit_X, WY)} F(X, WY)
      \ .
    \end{align*}
  \end{theorem}

  Before proving \Cref{prop:mixed_liftings}, let us record the following useful lemma.

  \begin{lemma}
    \label{prop:mixed_liftings_lemma}
    Let everything be as in \Cref{prop:mixed_liftings}, and assume that $F$ lifts.
    Also, let $(X, x)$ and $(Y, y)$ be a $V$-coalgebra and a $W$-algebra, respectively.
    Then
    \begin{enumerate}[leftmargin=0.8cm]
      \item 
        $
          F(\counit_X, Y) \in \cat^T 
          \big(
            \left(
              F(VX, Y),
              \triangleright_{\Delta_X, y}
            \right),\
            \left(
              F(X, Y), \triangleright_{x, y}
            \right)
          \big)
          $

      \item 
        $
          F(X, u^W_Y) \in \cat^T 
          \big(
            \left(
              F(X, Y),
              \triangleright_{x, y}
            \right),\
            \left(
              F(X, WY), \triangleright_{x, m^W_Y}
            \right)
          \big)
          $

      \item
        The natural transformation built from the lifting satisfies
        \begin{align*}
          \nu^\triangleright_{X, Y} 
          = F(\counit_X, u^W_Y) \circ \triangleright_{\Delta_X, y}
          = F(X, u^W_Y) \circ \triangleright_{x, y} \circ TF(\counit^X, Y)
          = \triangleright_{x, m^W_Y} \circ TF(\counit_X, u^W_Y)
        \end{align*}
    \end{enumerate}
  \end{lemma}
  \begin{proof}
    (1) follows immediately, since $F$ lifts by assumption and $\counit_X \colon (VX,
    \Delta_X) \to (X, x)$ as $V$-comodules.
    (2) is completely analogous.
    Finally, (3) follows from appropriate applications of (1) and (2).
  \end{proof}

  \begin{proof}[Proof of \Cref{prop:mixed_liftings}]
    We first establish the bijection.
    Let $\nu$ be given.
    Then
    \begin{align*}
      \nu^{\triangleright^\nu}_{X, Y}
      &= 
      F(\counit_X, W Y) 
      \circ \triangleright^\nu_{\Delta_X, m^W_Y}
      \circ TF (VX, u^W_Y)
      \\ &=
      F(\counit_X, W Y) 
      \circ F(VX, m^W_Y)
      \circ \nu_{VX, WY}
      \circ TF(\Delta_X, WY)
      \circ TF (VX, u^W_Y)
      \\ &=
      F(\counit_X, W Y) 
      \circ F(VX, m^W_Y)
      \circ \nu_{VX, WY}
      \circ TF (V^2 X, u^W_Y)
      \circ TF(\Delta_X, Y)
      \\ &=
      F(\counit_X, W Y) 
      \circ F(VX, m^W_Y)
      \circ TF (VX, W u^W_Y)
      \circ \nu_{VX, Y}
      \circ TF(\Delta_X, Y)
      \\ &=
      F(\counit_X, W Y) 
      \circ \nu_{VX, Y}
      \circ TF(\Delta_X, Y)
      \\ &=
      \nu_{X, Y}
      \circ TF(V \counit_X, Y) 
      \circ TF(\Delta_X, Y)
      \\ &=
      \nu_{X, Y}
      \ ,
    \end{align*}
    where we have used nothing but functoriality and the (co)unitality axioms.

    On the other hand, given $\triangleright$, we have
    \begin{align*}
      \triangleright^{\nu^\triangleright}_{x, y}
      &=
      F(X, y) 
      \circ \nu^\triangleright_{X, Y}
      \circ TF(x, Y)
      \\ &=
      F(X, y) 
      \circ F(\counit_X, WY)
      \circ \triangleright_{\Delta_X, m^W_Y}
      \circ TF(VX, u^W_Y)
      \circ TF(x, Y)
      \\ &\oversetEq[(i)]
      F(X, y) 
      \circ F(\counit_X, WY)
      \circ F(VX, u^W_Y)
      \circ \triangleright_{\Delta_X, y}
      \circ TF(x, Y)
      \\ &\oversetEq
      F(\counit_X, Y)
      \circ \triangleright_{\Delta_X, y}
      \circ TF(x, Y)
      \\ &\oversetEq[(ii)]
      \triangleright_{x, y}
      \circ TF(\counit_X, Y)
      \circ TF(x, Y)
      \\ &\oversetEq
      \triangleright_{x, y}
      \ .
    \end{align*}
    Both $(i)$ and $(ii)$ use \Cref{prop:mixed_liftings_lemma}.
    The other steps are again naturality and the (co)unitality axioms.

    Thus the bijection is established, and we are left with showing that everything is
    well-defined.
    Let first $\nu$ be given.
    We check unitality of $\triangleright^\nu$:
    \begin{align*}
      \triangleright^\nu_{x, x} \circ u^T_{F(X, Y)}
      &=
      F(X, y)
      \circ \nu_{X, Y}
      \circ TF(x, Y)
      \circ u^T_{F(X, Y)}
      \\ &=
      F(X, y)
      \circ \nu_{X, Y}
      \circ u^T_{F(VX, Y)}
      \circ F(x, Y)
      \\ &\oversetEq[(*)]
      F(X, y)
      \circ F(\counit_X, u^W_Y)
      \circ F(x, Y)
      \\ &=
      F(X, Y)
      \ ,
    \end{align*}
    where the only non-trivial step $(*)$ uses the mixed (co)unitality of $\nu$.
    To see the associativity of $\triangle^\nu$, we compute
    \begin{align*}
      \triangleright^\nu_{x, y}
      \circ m^T_{F(X, Y)}
      &=
      F(X, y)
      \circ \nu_{X, Y}
      \circ TF(x, Y)
      \circ m^T_{F(X, Y)}
      \\ &=
      F(X, y)
      \circ \nu_{X, Y}
      \circ m^T_{F(VX, Y)}
      \circ T^2 F(x, Y)
      \\ &\oversetEq[(i)]
      F(X, y)
      \circ F(X, m^W_Y)
      \circ \nu_{X, WY}
      \circ T \nu_{VX, Y}
      \circ T^2 F(\Delta_X, Y)
      \circ T^2 F(x, Y)
      \\ &\oversetEq[(ii)]
      F(X, y)
      \circ F(X, W y)
      \circ \nu_{X, WY}
      \circ T \nu_{VX, Y}
      \circ T^2 F(V x, Y)
      \circ T^2 F(x, Y)
      \\ &\oversetEq[(iii)]
      F(X, y)
      \circ F(X, W y)
      \circ \nu_{X, WY}
      \circ T F(x, W Y)
      \circ T \nu_{X, Y}
      \circ T^2 F(x, Y)
      \\ &\oversetEq[(iv)]
      F(X, y)
      \circ \nu_{X, Y}
      \circ T F(V X, y)
      \circ T F(x, W Y)
      \circ T \nu_{X, Y}
      \circ T^2 F(x, Y)
      \\ &\oversetEq
      F(X, y)
      \circ \nu_{X, Y}
      \circ T F(x, Y)
      \circ T 
      \big(
        F(X, y)
        \circ \nu_{X, Y}
        \circ T F(x, Y)
      \big)
      \\ &\oversetEq
      \triangleright^\nu_{x, y}
      \circ T \triangleright^\nu_{x, y}
      \ ,
    \end{align*}
    as desired.
    Here, $(i)$ uses mixed (co)associativity, $(ii)$ uses (co)associativity of (co)monad
    (co)actions, and $(iii)$ and $(iv)$ use naturality of $\nu$.

    Finally we need to check that for $(f, g) \colon ((X, x), (Y, y)) \to ((X', x'), (Y',
    y'))$, the morphism $F(f, g)$ indeed intertwines the $T$-actions.
    We have
    \begin{align*}
      F(f, g) \circ \triangleright^\nu_{x, y}
      &=
      F(f, g) 
      \circ F(X, y)
      \circ \nu_{X, Y}
      \circ TF(x, Y)
      \\ &=
      F(X', y')
      \circ F(f, W g) 
      \circ \nu_{X, Y}
      \circ TF(x, Y)
      \\ &=
      F(X', y')
      \circ \nu_{X', Y'}
      \circ TF(X f, g) 
      \circ TF(x, Y)
      \\ &=
      F(X', y')
      \circ \nu_{X', Y'}
      \circ TF(x', Y')
      \circ TF(f, g) 
      \\ &=
      \triangleright^\nu_{x', y'}
      \circ TF(f, g)
      \ ,
    \end{align*}
    exactly as needed.

    So now assume that $F$ lifts.
    In particular, the actions $\triangleright$ are given.
    We check the mixed (co)unitality of $\nu^\triangleright$:
    \begin{align*}
      \nu^\triangleright_{X, Y} 
      \circ u^T_{F(VX, Y)}
      &= 
      F(\counit_X, WY)
      \circ \triangleright_{\Delta_X, m^W_Y}
      \circ TF(VX, u^W_Y)
      \circ u^T_{F(VX, Y)}
      \\ &=
      F(\counit_X, WY)
      \circ \triangleright_{\Delta_X, m^W_Y}
      \circ u^T_{F(VX, WY)}
      \circ F(VX, u^W_Y)
      \\ &=
      F(\counit_X, u^W_Y)
      \ ,
    \end{align*}
    using unitality of $T$-actions.
    For the mixed (co)multiplicativity, we compute
    \begin{align*}
      &
      F(X, m^W_Y)
      \circ \nu^\triangleright_{X, WY}
      \circ T \nu^\triangleright_{VX, Y}
      \circ T^2 F(\Delta_X, Y)
      \\ &\oversetEq[(\star)]
      F(X, m^W_Y)
      \circ F(X, u^W_{WY})
      \circ \triangleright_{x, m^W_Y}
      \circ TF(\counit_X, WY)
      \circ T \nu^\triangleright_{VX, Y}
      \circ T^2 F(\Delta_X, Y)
      \\ &\oversetEq[(\star)]
      \triangleright_{x, m^W_Y}
      \circ TF(\counit_X, WY)
      \circ TF(VX, u^W_Y)
      \circ T \triangleright_{\Delta_X, y}
      \circ T^2 F(\counit_{VX}, Y)
      \circ T^2 F(\Delta_X, Y)
      \\ &\oversetEq
      \triangleright_{x, m^W_Y}
      \circ TF(\counit_X, WY)
      \circ TF(VX, u^W_Y)
      \circ T \triangleright_{\Delta_X, y}
      \\ &\oversetEq[(\star)]
      F(\counit_X, u^W_Y)
      \circ \triangleright_{\Delta_X, y}
      \circ T \triangleright_{\Delta_X, y}
      \\ &\oversetEq
      F(\counit_X, u^W_Y)
      \circ \triangleright_{\Delta_X, y}
      \circ m^T_{F(VX, Y)}
      \\ &\oversetEq[(\star)]
      \nu^\triangleright_{X, Y}
      \circ m^T_{F(VX, Y)}
      \ .
    \end{align*}
    Here, \Cref{prop:mixed_liftings_lemma} was used in all steps marked $(\star)$.

    Finally, we show that $\nu^\triangleright$ is natural.
    Let $(f, g) \colon (X, Y) \to (X', Y')$ in $\cat[D] \times \cat[E]$.
    Then
    \begin{align*}
      \nu^\triangleright_{X', Y'} \circ TF(Vf, g)
      &=
      F(\counit_{X'}, WY')
      \circ \triangleright_{\Delta_{X'}, m^W_{Y'}}
      \circ TF(VX', u^W_{Y'})
      \circ TF(Vf, g)
      \\ &=
      F(\counit_{X'}, WY')
      \circ \triangleright_{\Delta_{X'}, m^W_{Y'}}
      \circ TF(Vf, W g)
      \circ TF(VX, u^W_Y)
      \\ &\oversetEq[(*)]
      F(\counit_{X'}, WY')
      \circ F(Vf, W g)
      \circ \triangleright_{\Delta_{X}, m^W_{Y}}
      \circ TF(VX, u^W_Y)
      \\ &\oversetEq
      F(f, W g)
      \circ F(\counit_{X}, WY)
      \circ \triangleright_{\Delta_{X}, m^W_{Y}}
      \circ TF(VX, u^W_Y)
      \\ &\oversetEq
      F(f, W g) \circ \nu^\triangleright_{X, Y}
      \ ,
    \end{align*}
    where in $(*)$ we used that $(Vf, Wg)$ is a morphism in $\cat[D]^V \times \cat[E]^W$
    of (co)free (co)algebras, and thus $F(Vf, Wg)$ is a morphism in $\cat^T$, intertwining
    the actions $\triangleright_{\Delta_X, m^W_Y}$ and $\triangleright_{\Delta_{X'},
    m^W_{Y'}}$.
    This finishes the proof.
  \end{proof}
  
  Since a monad on a category is the same as a comonad on the opposite category, we
  immediately get the following.

  \begin{corollary}
    \label{prop:mixed_liftings_corollary}
    Let $(T, m, u)$ be a monad on a category $\cat$.
    Then a functor $F \colon \cat\op \times \cat \to \cat$ lifts to a functor $F^\#
    \colon (\cat^T)\op \times \cat^T \to \cat^T$ if and only if there is a natural
    transformation $\nu_{X, Y} \colon T F(TX, Y) \to F(X, TY)$ satisfying
    \begin{align*}
      \nu_{X, Y} \circ u_{F(TX, Y)} = F(u_X, u_Y)
    \end{align*}
    and
    \begin{align*}
      \nu_{X, Y} \circ m_{F(TX, Y)} 
      = F(X, m_Y) \circ \nu_{X, TY} \circ T \nu_{TX, Y} \circ T^2 F(m_X, Y)
      .
    \end{align*}
  \end{corollary}


\begin{thebibliography}{BFVV}
\bibitem[Be]{everybody}
G.M.\ Bergman, 
\emph{Everybody knows what a Hopf algebra is}.
Group actions on rings (Brunswick, Maine, 1984), 25--48.
Contemp. Math., 43
American Mathematical Society, Providence, RI, 1985.

  \bibitem[B\"o1]{Boehm-private}
  G.~B\"ohm,
  {\it private communication}.
	
	\bibitem[B\"o2]{Gabi-book}
	G.~B\"ohm,
	\emph{Hopf algebras and their generalizations from a category theoretical point of view}. 
	Lecture Notes in Mathematics, \textbf{2226}. Springer, Cham, 2018.

  \bibitem[BLV]{BLV-hopf_monads}
  A.~Brugui\`eres, S.~Lack, A.~Virelizier,
  {\it Hopf monads on monoidal categories}.
  Adv.\ in Math.\ \textbf{227} (2011), Issue 2,
  745--800.

\bibitem[BFVV]{BFVV}
M.~Buckley, T.~Fieremans, C.~Vasilakopoulou, J.~Vercruysse, {\it Oplax Hopf algebras}, High.\ Struct.\ {\bf 5} (2021), no. 1, 71--120.


  \bibitem[EK]{Eilenberg-Kelly}
  S.~Eilenberg, G.~M.~Kelly,
  {\it Closed categories}.
	1966 Proc. Conf. Categorical Algebra (La Jolla, Calif., 1965) pp. 421--562. Springer, New York.
	
	\bibitem[GNT]{Green-Nichols-Taft}
	J.~A.~Green, W.~D.~Nichols, E.~J.~Taft, 
	{\it Left Hopf algebras}.
  J.\ Algebra \textbf{65} (1980), no.\ 2, 399--411.
	
	\bibitem[K]{Kelly}
	G.~M.~Kelly, \emph{Basic concepts of enriched category theory}. Repr. Theory Appl. Categ. (2005), no. 10, vi+137 pp.
	
	\bibitem[LS2]{Lack-Street-formal}
	S.~Lack, R.~Street, \emph{The formal theory of monads. II}. Special volume celebrating the 70th birthday of Professor Max Kelly. J. Pure Appl. Algebra \textbf{175} (2002), no. 1-3, 243--265.

  \bibitem[LS1]{Lack-Street-skew-monoidales}
  S.~Lack, R.~Street,
  {\it Skew monoidales, skew warping and quantum categories}.
Theory Appl. Categ. {\bf 26} (2012), No. 15, 385--402.  
	
	\bibitem[LT]{Lauve-Taft}
	A.~Lauve, E.~J.~Taft, 
	{\it A class of left quantum groups modeled after $\mathrm{SL}_q(n)$}.
  J.\ Pure Appl.\ Algebra \textbf{208} (2007), no.\ 3, 797--803.
	
	\bibitem[L]{Leinster}
	T.~Leinster, \emph{Higher operads, higher categories}. London Mathematical Society Lecture Note Series, \textbf{298}. Cambridge University Press, Cambridge, 2004.
	
	\bibitem[Mc]{Maclane}
	S.~MacLane,
	{\it Categories for the working mathematician}.
	Graduate Texts in Mathematics, Vol. \textbf{5}. Springer-Verlag, New York-Berlin, 1971.
	
	\bibitem[Mj]{Majid}
	S.~Majid, \emph{Tannaka-Kre{\u\i}n theorem for quasi-Hopf algebras and other results}. Deformation theory and quantum groups with applications to mathematical physics (Amherst, MA, 1990), 219--232, Contemp. Math., \textbf{134}, Amer. Math. Soc., Providence, RI, 1992.
	
	\bibitem[NT]{Nichols-Taft}
	W.~D.~Nichols, E.~J.~Taft, 
	{\it The left antipodes of a left Hopf algebra},
  Algebraists' homage: papers in ring theory and related topics (New Haven, Conn., 1981),
  pp.\ 363--368, Contemp.\ Math., 13, Amer.\ Math.\ Soc., Providence, R.I., 1982.
	
	\bibitem[P]{Pareigis}
	B.~Pareigis, \emph{A noncommutative noncocommutative Hopf algebra in ``nature''}. J. Algebra \textbf{70} (1981), no. 2, 356--374.
	
	\bibitem[RT]{Rodriguez-Taft}
	S.~Rodr\'iguez-Romo, E.~J.~Taft, 
	{\it A left quantum group}. 
  J.\ Algebra \textbf{286} (2005), no.\ 1,
  154--160.
	
	\bibitem[Sa1]{Saracco-Tannaka}
	P.~Saracco, \emph{Coquasi-bialgebras with preantipode and rigid monoidal categories}. Algebr. Represent. Theory \textbf{24} (2021), no. 1, 55--80.
	
	\bibitem[Sa2]{Saracco-Frob}
	P.~Saracco,
  {\it Hopf modules, Frobenius functors and (one-sided) Hopf algebras}.
  J.\ Pure Appl.\ Algebra \textbf{225} (2021), no.\ 3,
  Paper No.\ 106537, 19 pp.
	
	\bibitem[Sc1]{Schauenburg-Tannaka}
	P.~Schauenburg, \emph{Tannaka duality for arbitrary Hopf algebras}.
Algebra Berichte [Algebra Reports], \textbf{66}. Verlag Reinhard Fischer, Munich, 1992.
	
	\bibitem[Sc2]{Schauenburg-bialgebroids}
	P.~Schauenburg, \emph{Bialgebras over noncommutative rings and a structure theorem for Hopf bimodules}. Appl. Categ. Structures \textbf{6} (1998), no. 2, 193--222.
	
	\bibitem[Sc3]{Schauenburg-Hopf}
	P.~Schauenburg, \emph{Duals and doubles of quantum groupoids ($\times_R$-Hopf algebras)}. New trends in Hopf algebra theory (La Falda, 1999), 273--299, Contemp. Math., \textbf{267}, Amer. Math. Soc., Providence, RI, 2000.

  \bibitem[St1]{Street-tricocycloids}
  R.~Street,
  {\it Fusion operators and cocycloids in monoidal categories}.
	Appl. Categ. Structures \textbf{6} (1998), no. 2, 177--191.

  \bibitem[St2]{Street-skew_closed}
  R.~Street,
  {\it Skew-closed categories}.
	J. Pure Appl. Algebra \textbf{217} (2013), no. 6, 973--988.
	
	\bibitem[Sz]{Szlachanyi-skew-monoidal}
  K.~Szlach\'anyi,
  {\it Skew-monoidal categories and bialgebroids}.
	Adv. Math. \textbf{231} (2012), no. 3-4, 1694--1730.
	
	\bibitem[U]{Ulbrich}
	K.-H.~Ulbrich, 
  {\it On Hopf algebras and rigid monoidal categories},
  Israel J. Math. \textbf{72} (1990), no. 1-2, 252--256.

  \bibitem[UVZ]{UVZ_Eilenberg-Kelly_reloaded}
  T.~Uustalu, N.~Veltri, N.~Zeilberger,
  {\it Eilenberg-Kelly reloaded}.
	The 36th Mathematical Foundations of Programming Semantics Conference, 2020, 233--256,
	Electron. Notes Theor. Comput. Sci., \textbf{352}, Elsevier Sci. B. V., Amsterdam, 2020.

\bibitem[V]{Ver}
J.~Vercruysse, Hopf algebras -- Variant notions and reconstruction theorems, in ``Quantum Physics and Linguistics: A Compositional, Diagrammatic Discourse'', C. Heunen, M. Sadrzadeh, E. Grefenstette (eds.). Oxford University Press, Oxford, 2013, 115--145.

\end{thebibliography}
\end{document}